\documentclass[11pt,a4paper]{amsart}

\usepackage{fixltx2e}
\usepackage[usenames,dvipsnames]{xcolor}
\usepackage{amsmath,amsfonts,amsbsy,amsgen,amscd,mathrsfs,amssymb,amscd}
\usepackage{amsthm}
\usepackage{url}
\usepackage[UKenglish]{babel}
\usepackage{eurosym}
\usepackage{tikz}
\usetikzlibrary{matrix,arrows,shapes}
\tikzset{mynode/.style = {
    circle,
    minimum size=2pt,
    draw=black, fill=black}
}
\usepackage{microtype}
\usepackage{enumitem}
\usepackage{listings}
\definecolor{darkblue}{rgb}{0,0,.75}
\lstnewenvironment{PseudoCode}[1][]
{\lstset{language=Matlab,basicstyle=\small, keywordstyle=\color{darkblue},numbers=none,xleftmargin=.04\textwidth,mathescape,frame=single,#1}}
{}

\definecolor{dark-gray}{gray}{0.3}
\definecolor{dkgray}{rgb}{.4,.4,.4}
\definecolor{dkblue}{rgb}{0,0,.5}
\definecolor{medblue}{rgb}{0,0,.75}
\definecolor{rust}{rgb}{0.5,0.1,0.1}

\usepackage[colorlinks=true]{hyperref}
\hypersetup{urlcolor=Blue}
\hypersetup{citecolor=Black}
\hypersetup{linkcolor=dark-gray}

\usepackage{graphicx}
\usepackage{booktabs,longtable,tabu} 
\usepackage{multirow} 
\usepackage{float}
\usepackage[T1]{fontenc}

\usepackage{times}
\usepackage{bm} 

\graphicspath{{figures/}}

\numberwithin{equation}{section} 

\providecommand{\mathbold}[1]{\bm{#1}}  

\newtheorem{theorem}{Theorem}[section]
\newtheorem{lemma}[theorem]{Lemma}

\newtheorem{proposition}[theorem]{Proposition}

\newtheorem{corollary}[theorem]{Corollary}

\theoremstyle{definition}

\newtheorem{example}[theorem]{Example}
\newtheorem{remark}[theorem]{Remark}

\renewcommand{\phi}{\varphi}

\newcommand{\e}{\varepsilon}

\renewcommand{\mid}{\mathrel{\mathop{:}}} 

\newcommand{\zerovct}{\vct{0}} 

\providecommand{\mathbbm}{\mathbb} 

\newcommand{\R}{\mathbbm{R}}

\newcommand{\polar}{\circ}

\newcommand{\diff}[1]{\mathrm{d}{#1}}

\newcommand{\argmin}{\operatorname*{arg\; min}}

\newcommand{\Prob}{\mathbbm{P}}

\newcommand{\Expect}{\operatorname{\mathbb{E}}}

\newcommand{\vct}[1]{\mathbold{#1}}
\newcommand{\mtx}[1]{\mathbold{#1}}

\newcommand{\Proj}{\ensuremath{\mtx{\Pi}}} 

\newcommand{\ip}[2]{\langle {#1}, {#2} \rangle}

\newcommand{\norm}[1]{\Vert {#1}\Vert}

\newcommand{\sdim}{\delta}

\newcommand{\veps}{\varepsilon}
\newcommand{\A}{\mathcal{A}}

\newcommand{\mF}{\mathcal{F}}

\newcommand{\mN}{\mathcal{N}}
\newcommand{\mL}{\mathcal{L}}
\newcommand{\mR}{\mathcal{R}}

\newcommand{\relint}{\operatorname{relint}}

\newcommand{\lin}{\operatorname{lin}}

\newcommand{\B}{\mathcal{B}}
\newcommand{\F}{\mathcal{F}}

\DeclareMathOperator{\vol}{vol}

\newcommand{\sttwo}[2]{\genfrac{\{}{\}}{0pt}{}{#1}{#2}}

\usepackage{3dplot}
\tdplotsetmaincoords{60}{110}
\title[Intrinsic Volumes of Polyhedral Cones]{Intrinsic Volumes of Polyhedral Cones: A combinatorial perspective}
\author{Dennis Amelunxen}
\address{Department of Mathematics\\ City University of Hong Kong\\ damelunx@cityu.edu.hk}
\author{Martin Lotz}
\address{School of Mathematics\\ The University of Manchester\\ martin.lotz@manchester.ac.uk}

\date{\today}



\begin{document}

\begin{abstract}
The theory of intrinsic volumes of convex cones has recently found striking applications in areas such as convex optimization and compressive sensing.
This article provides a self-contained account of the combinatorial theory of intrinsic volumes for polyhedral cones. Direct derivations of the General Steiner formula, the conic analogues of the Brianchon-Gram-Euler and the Gauss-Bonnet relations, and the Principal Kinematic Formula are given. In addition, a connection between the characteristic polynomial of a hyperplane arrangement and the intrinsic volumes of the regions of the arrangement, due to Klivans and Swartz, is generalized and some applications are presented.
\end{abstract}

\maketitle

\section{Introduction}
The theory of conic intrinsic volumes (or solid/internal/external/Grassmann angles) has a rich and varied history, with origins dating back at least to the work of Sommerville~\cite{sommerville1927relations}. This theory has recently found renewed interest, owing to newly found connections with measure concentration and resulting applications in compressive sensing, optimization, and related fields~\cite{dota:09a,amelunxen2014intrinsic,edge,mccoy2014steiner,goldstein2014gaussian}. 
Despite this recent surge in interest, the theory remains somewhat inaccessible to a general public in applied areas; this is, in part, due to the fact that many of the results are found using varying terminology (cf.~Section~\ref{sec:angles}), or are available as special cases of a more sophisticated theory of spherical integral geometry~\cite{scwe:08,Gl,sant:76} that treats the subject in a level of generality (involving curvature/support measures or relying on differential geometry) that is usually more than what is needed from the point of view of the above-mentioned applications.
In addition, some results, such as the relation to the theory of hyperplane arrangements, have so far not been connected to the existing body of research. 

One aim of this article is therefore to provide the practitioner with a self-contained account of the basic theory of intrinsic volumes of polyhedral cones that requires little more background than some elementary polyhedral geometry and properties of the Gaussian distribution. While some of the material is classic (see, for example,~\cite{mcmullen1975non}), we blend into the presentation a generalization of a formula of Klivans and Swartz~\cite{klivans2011projection}, with a streamlined proof and some applications.

The focus of this text is on simplicity rather than generality, on finding the most natural relations between different results that may be derived in different orders from each other, and on highlighting parallels between different results. Despite this, the text does contain some generalizations of known results, provided these can be derived with little additional effort. 
In the interest of brevity, this article does not discuss the probabilistic properties of intrinsic volumes, such as their moments and concentration properties, nor does it go into related geometric problems such as random projections of polytopes~\cite{vershik1992asymptotic,affentranger1992random}.

Section~\ref{sec:prelims} is devoted to some preliminaries from the theory of polyhedral cones including a discussion of conic intrinsic volumes, a section devoted to clarifying the connections between different notation and terminology used in the literature, and a section introducing some concepts and techniques from the theory of partially ordered sets. In Section~\ref{se:steiner} we present a modern interpretation of the conic Steiner formula that underlies the recent developments in~\cite{edge,mccoy2014steiner,goldstein2014gaussian}, and in Section~\ref{sec:Gauss-Bonnet}, which is based on the influential work of McMullen~\cite{mcmullen1975non}, we derive and discuss the Gauss-Bonnet relation for intrinsic volumes. Section~\ref{sec:kin-form} contains a crisp proof of the Principal Kinematic Formula for polyhedral cones, and Section~\ref{sec:Kl-Sw} is devoted to a generalization of a result by Klivans and Swartz~\cite{klivans2011projection} and some applications thereof.

\subsection{Notation and conventions}
Throughout, we use boldface letters for vectors and linear transformations. To lighten the notation we denote the set consisting solely of the zero vector by~$\vct0$.
We use calligraphic letters for families of sets. We use the notation $\subseteq$ for set inclusion and $\subset$ for strict inclusion.

\section{Preliminaries}\label{sec:prelims}

General references for basic facts about convex cones that are stated here are, for example,~\cite{Barv,Z:95,Rock}. More precise references will be given when necessary.
A convex cone $C\subseteq \R^d$ is a convex set such that $\lambda C=C$ for all $\lambda>0$. A convex cone is polyhedral if it is a finite intersection of closed half-spaces. In particular, linear subspaces are polyhedral, and polyhedral cones are closed. In what follows, unless otherwise stated, all cones are assumed to be polyhedral and non-empty. A supporting hyperplane of a convex cone $C$ is a linear hyperplane $H$ such that~$C$ lies entirely in one of the closed half-spaces induced by~$H$ (unless explicitly stated otherwise, all hyperplanes will be linear, i.e., linear subspaces of codimension one). A proper face of $C$ is a set of the form $F=C\cap H$, where $H$ is a supporting hyperplane. A set $F$ is called a face of $C$ if it is either a proper face or $C$ itself.
The linear span $\lin(C)$ of a cone~$C$ is the smallest linear subspace containing~$C$ and is given by $\lin(C)=C+(-C)$, where $A+B=\{\vct{x}+\vct{y}\mid \vct{x}\in A, \vct{y}\in B\}$ denotes the Minkowski sum of two sets $A$ and $B$. The dimension of a face $F$ is $\dim F:=\dim\lin(F)$, and the relative interior $\relint(F)$ is the interior of~$F$ in $\lin(F)$. A cone is pointed if the origin $\zerovct$ is a zero-dimensional face, or equivalently, if it does not contain a linear subspace of dimension greater than zero. If $C$ is not pointed, then it contains a nontrivial linear subspace of maximal dimension $k>0$, given by $L=C\cap (-C)$, and $L$ is contained in every supporting hyperplane (and thus, in every face) of $C$. Denoting by $C/L$ the orthogonal projection of $C$ on the orthogonal complement of $L$, the projection $C/L$ is pointed, and $C=L+C/L$ is an orthogonal decomposition of $C$; we call this the canonical decomposition of~$C$.

We denote by $\F(C)$ the set of faces, $\F_k(C)$ the set of $k$-dimensional faces, and let $f_k(C)=|\F_k(C)|$ denote the number of $k$-faces of $C$. The tuple $\vct f(C)=(f_0(C),\dots,f_d(C))$ is called the {\em $f$-vector} of $C$. Note that if $C=L+C/L$ is the canonical decomposition, then $\vct f(C)$ is a shifted version of $\vct f(C/L)$. The most fundamental property of the $f$-vector is the {\em Euler relation}.

\begin{theorem}[Euler]
Let $C\subseteq\R^d$ be a polyhedral cone. Then
\begin{equation}\label{Euler}
  \sum_{i=0}^d (-1)^i f_i(C) =
  \begin{cases}
     (-1)^{\dim L} & \text{ if } C=L \text{ is a linear subspace,}
  \\ 0 & \text{ else.}
  \end{cases}
\end{equation}
\end{theorem}

This relation is usually stated and proved in terms of polytopes~\cite[Ch.~8]{Z:95}, but intersecting a pointed cone with a suitable affine hyperplane yields a polytope with a face structure equivalent to that of the cone; the general case can be reduced to the pointed case through the canonical decomposition. A short proof of the Euler relation along with remarks on the history of this result can be found in~\cite{L:97}.

\subsection{Duality}\label{sec:duality}
The {\em polar} cone of a cone $C\subseteq\R^d$ is defined as
\begin{equation*}
 C^{\polar} = \{\vct{x}\in \R^d \mid \forall \vct{y}\in C, \ip{\vct{x}}{\vct{y}}\leq 0\}.
\end{equation*}
If $C=L$ is a linear subspace, then $C^{\polar}=L^{\perp}$ is just the orthogonal complement,
and the polar cone of the polar cone is again the original cone, as will be shown below.
To any face $F\in \F_k(C)$ we can associate the {\em normal face} $N_FC\in \F_{d-k}(C^{\polar})$ defined as $N_FC=C^\polar\cap \lin(F)^\bot$. To ease notation we will sometimes use $F^{\diamond}=N_FC$ when the cone is clear. The resulting map $\F_k(C)\to\F_{d-k}(C^\polar)$ is a bijection, which satisfies $N_{F^\diamond}(C^\polar)=F$.
This relation is easily deduced from the mentioned involution property of the polarity map, cf.~Proposition~\ref{prop:polarpolar} below.
The polar operation is order reversing, $C\subseteq D$ implies $C^\polar\supseteq D^\polar$, as follows directly from the definition; more properties will be presented below.

Central to convex geometry and optimization are a variety of theorems of the alternative, the most prominent of which is known as Farkas' Lemma (among the countless references, see for example~\cite[Chapter 2]{Z:95}). All versions of Farkas' Lemma follow from a special case of the Hahn-Banach theorem, the separating hyperplane theorem. In what follows we need a conic version of this result.

\begin{theorem}[Separating hyperplane for cones]\label{thm:sephyp}
 Let $C,D\subset \R^d$ be non-empty, closed convex cones. Then $\relint(C)\cap \relint(D)=\emptyset$ if and only if there exists a linear hyperplane $H$, not containing both $C$ and $D$, such that $C\subseteq H_+$ and $D\subseteq H_-$, where $H_+,H_-$ denote the closed half-spaces defined by $H$.
\end{theorem}

This theorem is usually stated for closed convex sets  and \textit{affine} hyperplanes $H$ (see, e.g.,~\cite[Theorem 11.3]{Rock}). Theorem~\ref{thm:sephyp} then follows from this more general version by noting that the relative interior of any non-empty, closed convex cone contains points arbitrary close to $\zerovct$, which implies $\zerovct\in H$.

The separating hyperplane theorem can be used to derive some interesting results involving the polar cone. The first such result states that polarity is an involution on the set of closed convex cones. We write $C^{\polar\polar}:=\left(C^{\polar}\right)^{\polar}$ for the polar of the polar.

\begin{proposition}\label{prop:polarpolar}
 Let $C$ be a non-empty, closed convex cone. Then $C^{\polar\polar} = C$.
\end{proposition}

\begin{proof}
Let $\vct{x}\in C$. Then, by definition of the polar, for all $\vct{y}\in C^{\polar}$ we have $\ip{\vct{x}}{\vct{y}}\leq 0$. This, in turn, implies that $\vct{x}\in C^{\polar\polar}$. Now let $\vct{x}\in C^{\polar\polar}$ and assume that $\vct{x}\not\in C$.
In particular, $\vct x\neq\vct0$, and by closedness of $C$ there exists $\veps>0$ such that the $\veps$-cone around $\vct x$, $B_\veps:=\{\vct y\mid \langle\vct x,\vct y\rangle \geq (1-\veps)\|\vct x\|\|\vct y\|\}$, satisfies $\relint(C)\cap \relint(B_\veps)=\emptyset$.
By Theorem~\ref{thm:sephyp}, there exists a hyperplane separating $C$ and $B_\veps$, and thus a non-zero $\vct{h}\in \R^d$ such that
\begin{equation*}
 \ip{\vct{x}}{\vct{h}}>0, \quad \forall \vct{y}\in C\mid \ip{\vct{h}}{\vct{y}}\leq 0.
\end{equation*} 
The first condition implies $\vct{h}\not\in C^{\polar}$, while the second one implies
 $\vct{h}\in C^{\polar}$. It follows that $\vct{x}\in C$.
\end{proof}

The following variation of Farkas' Lemma for convex cones, which is slightly more general than the usual one, is taken from~\cite{amelunxen2014gordon}. 

\begin{figure}[h!]
\begin{minipage}[c]{0.45\textwidth}
\centering
 \includegraphics[width=0.6\textwidth]{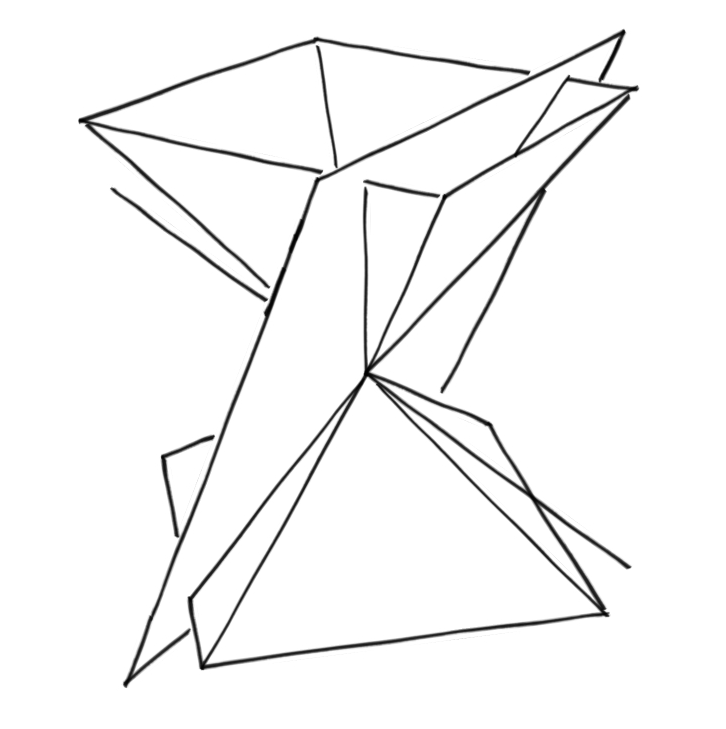}
 \end{minipage}
 \quad
 \begin{minipage}[c]{0.45\textwidth}
 \centering
 \includegraphics[width=0.7\textwidth]{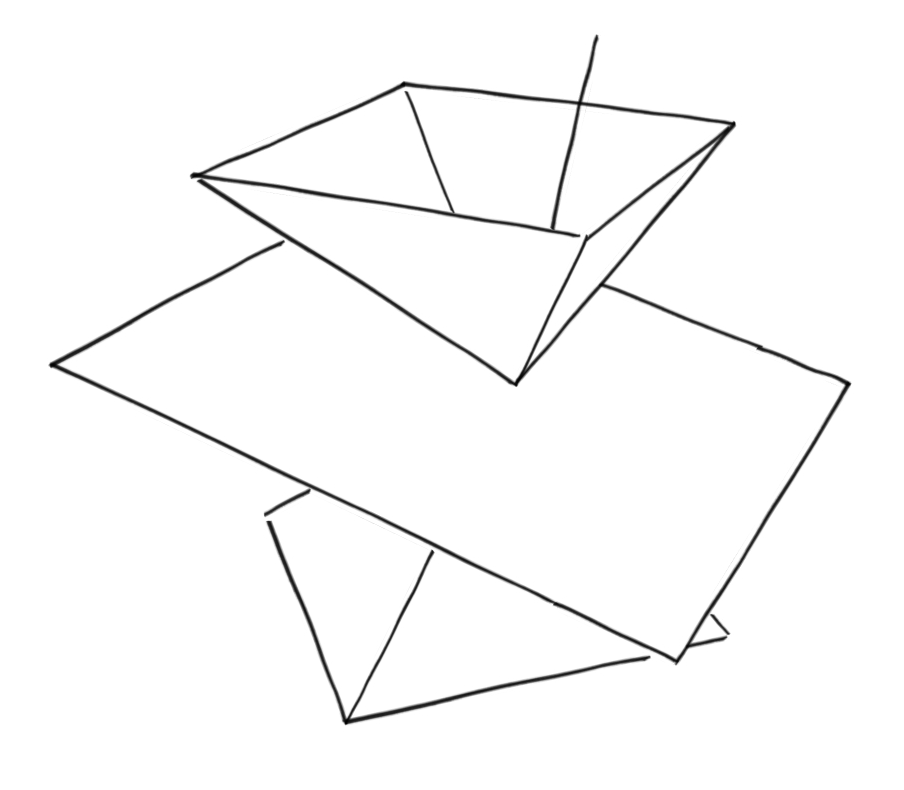}
 \end{minipage}
\caption{Either a subspace intersects $C$ non-trivially, or its complement intersects $C^{\circ}$.}
\label{fig:cone1}
\end{figure}

\begin{lemma}[Farkas]\label{le:farkas} Let $C,D$ be closed convex cones. Then
\begin{align}
\nonumber
 \relint(C)\cap D = \emptyset & \iff C^{\polar}\cap -D^{\polar} \neq \zerovct .
\intertext{In particular, if $D=L$ is a linear subspace, then}
\label{eq:farkas}
 \relint(C)\cap L = \emptyset & \iff C^{\polar}\cap L^{\perp} \neq \zerovct .
\end{align}
\end{lemma}

The situation in which $D=L$ is a hyperplane is best visualised as in Figure~\ref{fig:cone1}.

\begin{proof}
If $\relint(C)\cap D=\emptyset$, then by Theorem~\ref{thm:sephyp} there exists a separating hyperplane $H=\vct{h}^\bot$, $\vct h\neq\vct0$, such that $\langle \vct h,\vct x\rangle \leq 0$ for all $\vct x\in C$ and $\langle \vct h,\vct y\rangle \geq 0$ for all $\vct y\in D$. 
But this means $\vct h\in C^\polar\cap (-D^\polar)$.
On the other hand, if $\vct x\in \relint(C)\cap D$ then only in the case $C=\R^d$, for which the claim is trivial, can $\vct x=\vct0$ hold. If $\vct x\neq\vct0$, then~$C^\polar\setminus\vct0$ lies in the open half-space~$\{\vct h\mid \langle \vct h,\vct x\rangle<0\}$ and $-D^\polar$ lies in the closed half-space $\{\vct h\mid \langle \vct h,\vct x\rangle\geq0\}$, and thus $C^\polar\cap(-D^\polar)=\vct0$. The case $D=L$ follows immediately.
\end{proof}

In view of some of the later developments, it is important to understand the behaviour of duality under intersections. The following is a conic variant of~\cite[Corollary 23.8.1]{Rock} (see also~\cite[Chapter 7]{Z:95} for a similar theme).

\begin{proposition}\label{le:dualintersect}
The polar operation of intersection is the Minkowski sum,
\begin{equation*}
 (C\cap D)^{\polar}=C^{\polar}+D^{\polar}.
\end{equation*}
Moreover, every face of $C\cap D$ is of the form $F\cap G$ for some $F\in\mF(C),G\in\mF(D)$,
and the polar face satisfies
\begin{equation}\label{eq:N_(FcapG)(CcapD)supseteq...}
  N_{F\cap G}(C\cap D) \supseteq N_FC+N_GD. 
\end{equation}
If additionally $\relint(F)\cap\relint(G)\neq\emptyset$, then~\eqref{eq:N_(FcapG)(CcapD)supseteq...} holds with equality.
\end{proposition}

\begin{proof}
For the first claim, note that
\begin{align*}
C\cap D = C^{\polar\polar}\cap D^{\polar\polar}&=\{\vct{z}\in \R^d\mid \forall (\vct{x},\vct{y})\in C^{\polar}\times D^{\polar},\ip{\vct{z}}{\vct{x}}\leq 0, \ip{\vct{z}}{\vct{y}}\leq 0\}\\
 &= \{\vct{z}\in \R^d\mid \forall (\vct{x},\vct{y})\in C^{\polar}\times D^{\polar}, \ip{\vct{z}}{\vct{x}+\vct{y}}\leq 0\}= (C^{\polar}+D^{\polar})^{\polar}.
\end{align*}
where in the first equality we used Proposition~\ref{prop:polarpolar}; the third equality is easily verified by noting that $\ip{\vct{z}}{\vct{x}+\vct{0}}=\ip{\vct{z}}{\vct{x}}$ and $\ip{\vct{z}}{\vct{0}+\vct{y}}=\ip{\vct{z}}{\vct{y}}$. The first claim then follows by polarity and another application of Proposition~\ref{prop:polarpolar}, noting that the Minkowski sum of polyhedral cones is a polyhedral cone.

For the second claim, note that a face $\bar F\in \mF(C\cap D)$ can be written as $\bar F = \{\vct x\in C\cap D\mid \langle\vct x,\vct{h}\rangle=0\}$ for some $\vct{h}\in (C\cap D)^\polar$. By the first claim, we can write the normal vector in the form $\vct{h}=\vct{h}_C+\vct{h}_D$ with $\vct{h}_C\in C^\polar$ and $\vct{h}_D\in D^\polar$. Denoting $F:=\{\vct x\in C\mid \langle\vct x,\vct{h}_C\rangle=0\}\in\mF(C)$, $G:=\{\vct y\in D\mid \langle\vct y,\vct{h}_D\rangle=0\}\in\mF(D)$, we obtain
\begin{align*}
   F\cap G & = \{\vct x\in C\cap D\mid \langle\vct x,\vct{h}_C\rangle=\langle\vct x,\vct{h}_D\rangle=0\}
= \{\vct x\in C\cap D\mid \langle\vct x,\vct{h}\rangle=0\} = \bar F ,
\end{align*}
where the second equality follows from the fact that $\langle\vct x,\vct{h}_C\rangle\leq 0$ and $\langle\vct x,\vct{h}_D\rangle\leq 0$ if $\vct x\in C\cap D$.

Finally, for the claim about the polar face, note that, by what we have just shown and using double polarity, 
\begin{equation*}
 (N_FC)^\polar=(C^\polar\cap \lin(F)^\bot)^\polar=C+\lin(F)=C+(-F),
\end{equation*}
so that
\begin{align}\label{eq:(N_(FcapG)(CcapD))polar=...}
  (N_{F\cap G}(C\cap D))^\polar = (C\cap D)+(-(F\cap G)) & \subseteq (C+(-F))\cap (D+(-G))
\\ & = (N_FC)^\polar\cap(N_GD)^\polar = (N_FC+N_GD)^\polar .
\nonumber
\end{align}
The claim~\eqref{eq:N_(FcapG)(CcapD)supseteq...} follows by invoking polarity again.

To show that the inclusion in the above display is an equality if $\relint(F)\cap \relint(G)\neq \emptyset$, note first that if $\vct x\in \relint(F)$, then for every $\vct y\in C+(-F)$ we have $\vct y+\lambda\vct x\in C$ for $\lambda>0$ large enough. Indeed, if $\vct y=\vct y_C-\vct y_F$ with $\vct y_C\in C$, $\vct y_F\in F$, then $\vct y+\lambda\vct x = \vct y_C+\lambda(\vct x-\frac{1}{\lambda}\vct y_F)$, and $\vct x-\frac{1}{\lambda}\vct y_F\in F$ for $\lambda>0$ large enough. Now, if $\vct x\in \relint(F)\cap \relint(G)$ and $\vct y\in (C+(-F))\cap (D+(-G))$, then for $\lambda>0$ large enough, $\vct y+\lambda\vct x\in C\cap D$. Hence,
  \[ \vct y = (\vct y+\lambda\vct x) + (-\lambda\vct x) \in (C\cap D) + (-(F\cap G)) , \]
which shows that~\eqref{eq:(N_(FcapG)(CcapD))polar=...}, and thus~\eqref{eq:N_(FcapG)(CcapD)supseteq...}, hold with equality.
\end{proof}

Two faces $F\in \F(C)$ and $G\in \F(D)$ are said to intersect transversely, written $F\pitchfork G$, if their relative interiors have a non-empty intersection, $\relint(F)\cap \relint(G)\neq \emptyset$, and $\dim F\cap G=\dim F+\dim G-d$.

\begin{corollary}\label{cor:dualintersect}
Let $C,D$ be cones and $F\in \F(C)$, $G\in \F(D)$ be faces that intersect transversely. Then $N_FC+N_GD=N_{F\cap G}(C\cap D)$, and is a face of $C^{\polar}+D^{\polar}$ of dimension $(d-\dim F)+(d-\dim G)$.
\end{corollary}

For a polyhedral cone $C\subseteq\R^d$, denote by $\Proj_C$ the Euclidean projection,
\begin{equation}\label{eq:orthodecom}
 \Proj_C(\vct{x}) = \argmin \{\norm{\vct{x}-\vct{y}}^2 \mid \vct{y}\in C\} .
\end{equation}
The Moreau decomposition of a point $\vct{x}\in \R^d$ is the sum representation
\begin{equation}\label{eq:moreau}
 \vct{x} = \Proj_C(\vct{x})+\Proj_{C^{\circ}}(\vct{x}),
\end{equation}
where $\Proj_C(\vct{x})$ and $\Proj_{C^{\circ}}(\vct{x})$ are orthogonal. 
A direct consequence is the disjoint decomposition 
\begin{equation}\label{eq:conedecomp}
 \R^d = \bigcup_{F\in \mathcal{F}(C)} (\relint(F)+N_FC),
\end{equation}
see also~\cite[Lemma 3]{mcmullen1975non}.

\subsection{Intrinsic volumes}\label{sec:intr-vols}
For $C\subseteq\R^d$ a polyhedral cone and for two faces $F,G\in \F(C)$, define
\begin{equation*}
 v_F(G) = \Prob\{\Proj_G(\vct{g})\in \relint F\},
\end{equation*}
where $\vct{g}\sim\mN(\R^d)$ is a standard Gaussian vector in~$\R^d$.  If $F\subseteq G$, it follows from~\eqref{eq:conedecomp} that
\begin{equation*}
  v_{F}(G) = \Prob\{\vct g\in F+N_FG\}.
\end{equation*}
On the other hand, since the relative interiors of faces of $C$ are disjoint,
we have $v_F(G)=0$ if $F\not\subseteq G$. For the most part we will consider the case $G=C$. Define the $k$-th intrinsic volumes of $C$, $0\leq k\leq d$, to be 
\begin{equation*}
 v_k(C)=\sum_{F\in \F_k(C)} v_F(C).
\end{equation*}
For a fixed cone, the intrinsic volumes form a probability distribution on~$\{0,1,\ldots,d\}$.
Note that if $F\in\mF_k(C)$ then,
by the decomposition~\eqref{eq:moreau},
  \[ v_F(C) = v_k(F)\,v_{d-k}(N_FC) . \]
For later reference, we note that in combination with Corollary~\ref{cor:dualintersect}, we get for cones $C,D$ and faces $F\in \F_k(C)$, $G\in \F_{\ell}(D)$ that intersect transversely, with $j=k+\ell-d$,
\begin{equation}\label{eq:intvolintersect}
 v_{F\cap G}(C\cap D) = v_j(F\cap G)\,v_{d-j}(N_FC+N_GD) .
\end{equation}

\begin{example}
 Let $C=L\subseteq V$ be a linear subspace of dimension $i$. Then
 \begin{equation*}
  v_k(C) = \begin{cases}
            1 & \text{ if } k=i,\\
            0 & \text{ if } k\neq i.
           \end{cases}
 \end{equation*}
\end{example}

\begin{example}
 Let $C=\R^d_{\geq 0}$ be the non-negative orthant, i.e., the cone consisting of points with non-negative coordinates. A vector $\vct{x}$ projects orthogonally to a $k$-dimensional face of $C$ if and only if exactly $k$ coordinates are non-positive. By symmetry considerations and the invariance of the Gaussian distribution under permutations of the coordinates, it follows that
 \begin{equation*}
  v_k(\R^d_{\geq 0}) = \binom{d}{k}2^{-d}.
 \end{equation*}
\end{example}

The following important properties of the intrinsic volumes, which are easily verified in the setting of polyhedral cones, will be used frequently:
\begin{itemize}
 \item[(a)] {\bf Orthogonal invariance.} For an orthogonal transformation $\vct Q\in O(d)$,
 \begin{equation*}
  v_k(\vct QC) = v_k(C);
 \end{equation*}
 \item[(b)] {\bf Polarity.} 
 \begin{equation*}
  v_k(C) = v_{d-k}(C^{\polar});
 \end{equation*}
 \item[(c)] {\bf Product rule.} 
 \begin{equation}\label{eq:productrule}
  v_k(C\times D) = \sum_{i+j=k} v_i(C)v_j(D).
 \end{equation}
\end{itemize}

Note that the product rule implies $v_i(C\times L)=v_{i-k}(C)$ if $i\geq k$ and $L$ is a subspace of dimension $k$.
We will sometimes be working with the intrinsic volume generating polynomial,
\begin{equation*}
 P_C(t) = \sum_{k=0}^d v_k(C) t^k.
\end{equation*}
The product rule then states that the generating polynomial is multiplicative with respect to direct products. A direct consequence of the orthogonal invariance and the polarity rule is that the intrinsic volume sequence is symmetric for self-dual cones (i.e., cones such that $C=-C^{\polar}$).

An important summary parameter is the expected value of the distribution associated to the intrinsic volumes, the {\em statistical dimension}, which coincides with the expected squared norm of the projection of a Gaussian vector on the cone,
\begin{equation*}
 \delta(C) = \sum_{k=0}^d kv_k(C) = \Expect\big[\norm{\Proj_C(\vct{g})}^2\big] .
\end{equation*}
The statistical dimension reduces to the usual dimension for linear subspaces. The coincidence of the two expected values is a special case of the generalized Steiner formula~\ref{thm:gensteiner}, and is crucial in applications of the statistical dimension. More on the statistical dimension and its properties and applications can be found in~\cite{edge,mccoy2014steiner,goldstein2014gaussian}.

\subsection{Angles}\label{sec:angles}
In the classical works on polyhedral cones, intrinsic volumes were viewed as polytope angles, see \cite{feldman2009angles} for some perspective. 
Polyhedral cones arise as tangent or normal cones of polyhedra $K\subseteq \R^d$. Given such a polyhedron $K$ and a face $F\subseteq K$, with $\vct{x}_0\in \relint(F)$, the {\em tangent cone} $T_FK$ is defined as
\begin{equation*}
 T_FK = \bigcup_{\tau>0} \{\vct{v} \in \R^d \mid \vct{x}_0+\tau\vct{v}\in K\}.
\end{equation*}
The {\em normal cone} to $K$ at $F$ is the polar of the tangent cone.
To clarify the relations to the terminology used in this paper and to facilitate a translation of the results of some of the referenced papers, we provide the following list.

\subsubsection{Solid angle}

When speaking about the solid angle of a cone $C\subseteq\R^d$, sometimes denoted $\alpha(C)$, one usually assumes that $C$ has nonempty interior, and one defines $\alpha(C)$ as the Gaussian volume of $C$ (or equivalently, the relative spherical volume of $C\cap S^{d-1}$, where $S^{d-1}$ is the $(d-1)$-dimensional unit sphere); we extend this definition to also cover lower-dimensional cones, and define for $\dim C=k$,
\[ \alpha(C) := v_C(C) = v_k(C) = v_{d-k}(C^\polar). \]

\subsubsection{Internal/external angle}\label{sub:extangle}

The internal and external angle of a polyhedral set $K\subseteq\R^d$ at a face $F$ are defined as the solid angle of the tangent and normal cone of $K$ at $F$, respectively,
\begin{align*}
   \beta(F,K) & = \alpha(T_FK) , & \gamma(F,K) & = \alpha(N_FK) .
\end{align*}
Note that we have $v_F(C) = \beta(\vct0,F)\gamma(F,C)$.
Furthermore, conic polarity swaps between internal and external angles:
\begin{align*}
   \beta(F,C) & = \gamma(F^\diamond,C^\polar) , & \gamma(F,C) & = \beta(F^\diamond,C^\polar) ,
\end{align*}
where we use the notation $F^\diamond:=N_FC$ for the face of $C^\polar$, which is polar to the face $F$ of $C$. This shows that any formula involving the internal and external angles of a cone~$C$ has a polar version in terms of the internal and external angles of~$C^\polar$ where the roles of internal and external have been exchanged. (Some of the formulas in~\cite{mcmullen1975non} are stated in this polar version.)

\begin{remark}
The Brianchon-Gram-Euler relation~\cite[Thm.~(1)]{PS:67} of a convex polytope $K$ translates in the above notation as
  \[ \sum_{F\in\F(K)} (-1)^{\dim F} \beta(F,K) = 0 . \]
Replacing the bounded polytope by an unbounded cone makes this relation invalid. However, there exists a closely related conic version, which is called Sommerville's Theorem~\cite[Thm.~(37)]{PS:67}. This in turn can be used to derive a Gauss-Bonnet relation, cf.~Section~\ref{sec:Gauss-Bonnet}.
\end{remark}

\subsubsection{Grassmann angle}\label{sec:grassm-ang}
The Grassmann angles of a cone $C$, which have been introduced and analyzed by Gr\"unbaum~\cite{G:68}, are defined through the probability that a uniformly random linear subspace of a specific (co)dimension intersects the cone nontrivially. The kinematic/Crofton formulae express this probability in terms of the intrinsic volumes, cf.~Section~\ref{sec:kin-form}. More precisely, we have
\begin{equation}\label{eq:Croft-subsp}
  \Prob\{C\cap L_k\neq\vct0\} = 2\sum_{i\geq1 \text{ odd}} v_{k+i}(C) =: 2h_{k+1}(C) ,
\end{equation}
where $L_k\subseteq\R^d$ denotes a uniformly random linear subspace of codimension~$k$. Notice that when considering the intrinsic volumes and the Grassmann angles as vectors, $(v_0,v_1,\ldots,v_d)$ and $(h_0,h_1,\ldots,h_d)$, then these are related through a nonsingular linear transformation. Hence, any formula in the intrinsic volumes of a cone has an equivalent form in terms of Grassmann angles and vice versa; in this paper we prefer the intrinsic volume versions.

\begin{remark}
The preference of intrinsic volumes over Grassmann angles has an odd effect on the logic behind Corollary~\ref{thm:1} below, which is attributed to Gr\"unbaum. This result is originally stated and proved in~\cite[Thm.~2.8]{G:68} in terms of the Grassmann angles. So in order to rewrite Corollary~\ref{thm:1} in its original form, one needs to apply Crofton's formula~\eqref{eq:Croft-subsp} whose proof, given in Section~\ref{sec:kin-form}, uses Gauss-Bonnet~\eqref{cor:GB}, which in turn is a direct consequence of Corollary~\ref{thm:1}. The resulting proof of the original result~\cite[Thm.~2.8]{G:68} (in terms of Grassmann angles) is thus much less direct than the original one given by Gr\"unbaum.
\end{remark}

\subsection{Some poset techniques}\label{sec:posets}

In this section we recall some notions from the theory of partially ordered sets (posets) that we will need in Section~\ref{sec:Kl-Sw}. We only recall those properties that we will directly use, see~\cite[Ch.~3]{EC} for more details and context.

A lattice is a poset with the property that any two elements have both a least upper bound and a greatest lower bound. We will only consider finite lattices; in particular, for these lattices the greatest and the least elements $\hat1,\hat0$ both exist. More precisely, we will consider the following two (types of) finite lattices.

\begin{example}[Face lattice]\label{ex:face-latt}
Let $C\subseteq\R^d$ be a polyhedral cone. Then the set of faces $\F(C)$ with partial order given by inclusion is a finite lattice. The elements $\hat1,\hat0$ are given by $\hat1=C$ and $\hat0=C\cap (-C)$.
\end{example}

\begin{example}[Intersection lattice of a hyperplane arrangement]\label{ex:inters-latt}
Let $\A=\{H_1,\ldots,H_n\}$ be a set of (linear) hyperplanes $H_i\subset\R^d$, $i=1,\ldots,n$. The set of all intersections $\mL(\A)=\big\{\bigcap_{i\in I} H_i\mid I\subseteq\{1,\ldots,n\}\big\}$, endowed with the partial order given by reverse
inclusion, is called the intersection lattice of the hyperplane arrangement~$\A$. This lattice has a disjoint decomposition into $\mL_0(\A),\dots,\mL_d(\A)$, where $\mL_j(\A)=\{L\in \mL(\A)\mid \dim L=j\}$. The minimal and maximal elements are given by $\hat0=\R^d$ and $\hat1=\bigcap_{i=1}^n H_i$.
\end{example}

One can define the (real) incidence algebra of a (locally) finite poset $(P,\preceq)$ as the set of all functions $\xi\colon P\times P\to\R$, which besides having the usual vector space structure also possesses the multiplication
\begin{align*}
   \xi \nu & \colon P\times P\to \R , & \xi \nu(x,y) = \sum_{x\preceq z\preceq y} \xi(x,z)\,\nu(z,y) 
\end{align*}
defined for two functions $\xi,\nu\colon P\times P\to \R$.
The identity element in this algebra is the Kronecker delta, $\delta(x,y)=1$ if $x=y$ and $\delta(x,y)=0$ else. Another important element is the characteristic function of the partial order, $\zeta(x,y)=1$ if $x\preceq y$ and $\zeta(x,y)=0$ else. This function is invertible, and its inverse~$\mu$, called \emph{M\"obius function} on~$P$, can be recursively defined by $\mu(x,y)=0$ if $x\not\preceq y$, and
\begin{align}\label{eq:def-moeb}
   \mu(x,x) & = 1 , & \mu(x,y) & = -\sum_{x\preceq z\prec y} \mu(x,z) \quad \text{if } x\prec y .
\end{align}
The incidence algebra acts on the set of functions $f\colon P\to \R$ on the right by
\begin{equation*}
 (f\xi)(y) = \sum_{x\preceq y} f(x)\xi(x,y).
\end{equation*}
The M\"obius inversion is the simple fact that for two functions $f,g\colon P\to \R$ one has $f\zeta=g$ if and only if $f=g\mu$. Explicitly, this equivalence can be written out as follows:
\begin{equation}\label{eq:moeb-inv}
   \forall y\in P: g(y)=\sum_{x\preceq y} f(x) \quad\iff\quad \forall y\in P: f(y)=\sum_{x\preceq y} g(x)\mu(x,y) .
\end{equation}
The M\"obius function of the face lattice from Example~\ref{ex:face-latt} is given by $\mu(F,G)=(-1)^{\dim G-\dim F}$. For a whole range of techniques for computing M\"obius functions we refer to~\cite{EC,ardila2014algebraic}.

\subsubsection{Some elementary facts about hyperplane arrangements}

The last concept we need to introduce is that of a characteristic polynomial, which can be defined for any finite graded lattice; we only introduce the characteristic polynomial for hyperplane arrangements, as we will only use it in this context. We use the notation from Example~\ref{ex:inters-latt}. The characteristic polynomial of a hyperplane arrangement~$\A$ in~$\R^d$ is defined as~\cite[Sec.~3.11.2]{EC}
\begin{equation*}
 \chi_\A(t) = \sum_{L\in \mL(\A)}\mu(\R^d,L)t^{\dim L} .
\end{equation*}
More generally, we introduce the $j$th-level characteristic polynomial of~$\A$ as follows,
\begin{equation}\label{eq:def-chi_(A,k)(t)}
 \chi_{\A,j}(t) = \sum_{\tilde L\in \mL_j(\A)}\,\sum_{L\in \mL(\A)}\mu(\tilde L,L)t^{\dim L} ,
\end{equation}
so that $\chi_\A=\chi_{\A,d}$, and we also introduce the bivariate polynomial\footnote{This bivariate polynomial (or simple transformations thereof) is also known as M\"obius polynomial~\cite{Z:75} or Whitney polynomial~\cite{A:thesis,A:96}; it should not be confused with the coboundary/Tutte polynomial~\cite{J:12}.}
\begin{equation}\label{eq:biv-char}
  X_\A(s,t) := \sum_{j=0}^d s^j \chi_{\A,j}(t) = \sum_{\tilde L,L\in \mL(\A)}\mu(\tilde L,L)s^{\dim \tilde L} t^{\dim L} .
\end{equation}
The $j$th level characteristic polynomial can be written in terms of characteristic polynomials by considering restrictions of~$\A$: If $L\subseteq\R^d$ is a linear subspace, then the arrangement $\A^L=\{H\cap L\mid H\in\A, L\not\subseteq H\}$ is a hyperplane arrangement relative to~$L$. It is easily seen that we obtain
\begin{equation}\label{eq:expre-chi_(A,k)-chi_A}
  \chi_{\A,j}(t) = \sum_{L\in \mL_j(\A)}\,\chi_{\A^L}(t) .
\end{equation}
The M\"obius function of the intersection lattice alternates in sign~\cite[Prop.3.10.1]{EC}, and so do the coefficients of the ($j$th-level) characteristic polynomial. Note that $\chi_{\A,j}(t)$ (is either zero or) has degree~$j$ and the leading coefficient is given by~$|\mL_j(\A)|=:\ell_j(\A)$. For future reference we also note that in the cases $j=0,1$ we have
\begin{align}\label{eq:chi_(A,0/1)(t)}
   \chi_{\A,0}(t) & = \ell_0(\A), & \chi_{\A,1}(t) & = \ell_1(\A) (t-\ell_0(\A)) .
\end{align}

The complement of the hyperplanes of an arrangement $\A$, $\R^d\setminus \bigcup_{H\in \A} H$, decomposes into open convex cones. We denote by $\mR(\A)$ the set of polyhedral cones given by the closures of these regions, and we denote $r(\A):=|\mR(\A)|$. More generally, we define
\begin{align}\label{eq:def-R_k(A)}
   \mR_j(\A) & = \bigcup_{C\in \mR(\A)} \F_j(C) , & r_j(\A) & = |\mR_j(\A)| ,
\end{align}
so that $\mR(\A)=\mR_d(\A)$ and $r(\A)=r_d(\A)$.
The following theorem by Zaslavsky~\cite{Z:75}
lies at the heart of the result by Klivans and Swartz~\cite{klivans2011projection} that we will present in Section~\ref{sec:Kl-Sw}.

\begin{theorem}[Zaslavsky]\label{thm:zas1}
Let $\A$ be an arrangement of linear hyperplanes in $\R^d$. Then 
\begin{equation*}
  r_j(\A) = (-1)^j\,\chi_{\A,j}(-1) .
\end{equation*}
\end{theorem}

Note that since the coefficients of the characteristic polynomial alternate in sign, the number of $j$-dimensional regions, $r_j(\A)$, is given by the sum of the absolute values of the coefficients of the $j$th-level characteristic polynomial.

\section{The conic Steiner formula}\label{se:steiner}

A classic result in integral geometry is the Steiner Formula: the $d$-dimensional measure of the $\e$-neighbourhood of a convex body $K\subset\R^d$ (compact, convex) can be expressed as a polynomial in $\e$ of degree at most $d$, with the {\em intrinsic volumes} as coefficients:
\begin{equation}\label{eq:Steiner-eucl}
 \vol (K+\e B^d) = \sum_{i=0}^d V_i(C) \omega_{d-i} \e^{d-i},  
\end{equation}
where $B^d$ denotes the unit ball, $\omega_{d-i}=\vol(B^{d-i})=\frac{2\pi^{(d-i)/2}}{\Gamma((d-i)/2+1)}$, and the $V_i(K)$ are the Euclidean intrinsic volumes (see, e.g.,~\cite[Theorem 9.2.3]{KR:97}). For example, in the two-dimensional case, we have the situation of Figure~\ref{fig:steiner}. 

\begin{figure}
\begin{tikzpicture}[scale=2]
\def\myEps{0.25}
\draw[thick] (0,0) -- (1,0) -- (0,1) -- (0,0);
\draw[densely dotted] (0,0) -- ++(180:\myEps) (0,0) -- ++(270:\myEps)
              (1,0) -- ++(45:\myEps) (1,0) -- ++(270:\myEps)
              (0,1) -- ++(180:\myEps) (0,1) -- ++(45:\myEps);
\draw (180:\myEps) arc(180:270:\myEps) -- ++(1,0) arc(270:360:\myEps) arc(0:45:\myEps) -- ++(135:1.4142) arc(45:180:\myEps) -- ++(0,-1);
\end{tikzpicture}
\caption{$\mathrm{vol}(K+\e B^2) = \text{ area} + \text{ circumference} \cdot \e + \pi\cdot \e^2$}
\label{fig:steiner}
\end{figure}

In order to state an analogous result for convex cones or spherically convex sets (intersections of convex cones with the unit sphere), we have to agree on a notion of distance. A natural choice here is the capped angle $\sphericalangle(C,\vct{x})=\arccos(\|\Proj_C(\vct{x})\|/\|\vct{x}\|)$. 
Note that with this definition, $\sphericalangle(C,\vct{x})\leq \pi/2$, and is equal to $\pi/2$ if and only if $\vct{x}\in C^{\circ}$.
Note also that for $\vct{x}$ with $\norm{\vct{x}}=1$ and $\alpha\leq \pi/2$, we have $\sphericalangle(C,\vct{x})\leq \alpha$ if and only if $\norm{\Pi_C(\vct{x})}^2\geq \cos^2(\alpha)$.
Using this notion of distance, one obtains a formula similar to the Euclidean Steiner formula~\eqref{eq:Steiner-eucl}, which is usually called \emph{spherical/conic Steiner formula}, see for example~\cite[Chapter 6.5]{scwe:08} and the references given there, or the formula below.

It turns out that, when working with cones rather than spherically convex sets, it is convenient to work with the squared length of the projection on $C$ rather than with the angle. Moreover, it turns out quite useful to also consider the squared length of the projection on the polar cone $C^\polar$. The following general Steiner formula in the conic setting is due to McCoy and Tropp~\cite[Theorem 3.1]{mccoy2014steiner}; its formulation in probabilistic terms, as suggested by Goldstein, Nourdin and Peccati~\cite{goldstein2014gaussian}, is remarkably elegant. We sketch their proof (in the polyhedral case) below.

\begin{theorem}\label{thm:gensteiner}
Let $C\subseteq\R^d$ be a convex polyhedral cone, let $\vct g\in\R^d$ be a Gaussian vector, and let the discrete random variable $V$ on $\{0,1,\ldots,d\}$ be given by $\Prob\{V=k\}=v_k(C)$. Then
\begin{equation}\label{eq:Steiner-conic}
  \big( \norm{\Pi_C(\vct{g})}^2,\norm{\Pi_{C^\polar}(\vct{g})}^2\big) \stackrel{d}{=} \big( X_V,Y_{d-V}\big)
\end{equation}
where $\stackrel{d}{=}$ denotes equality in distribution, and $X_k,Y_k$ are independent $\chi^2$-distributed random variables with $k$ degrees of freedom.
\end{theorem}

A geometric interpretation of this form of the conic Steiner formula is readily obtained by considering \emph{moments} of the two sides in~\eqref{eq:Steiner-conic}. Indeed, the expectation of $f\big( \norm{\Pi_C(\vct{g})}^2,\norm{\Pi_{C^\polar}(\vct{g})}^2\big)$ equals the Gaussian volume of the angular neighbourhood around $C$ of radius~$\alpha\leq\frac{\pi}{2}$,
i.e., of the set $T_{\alpha}(C):=\{\vct x\mid \sphericalangle(C,\vct x)\leq\alpha\}$,
if one sets $f(x,y)=1$ if $x/(x+y)\geq\cos^2\alpha$, and $f(x,y)=0$ otherwise. For this choice of~$f$ the expectation of $f\big( X_V,Y_{d-V}\big)$ becomes a finite sum $\sum_{k=0}^d v_k(C) \Prob\{\vct g\in T_\alpha(L_k)\}$, where $T_\alpha(L_k)$ denotes the angular neighbourhood of radius~$\alpha$ around a $k$-dimensional linear subspace.
These Gaussian volumes of angular neighborhoods of linear subspaces replace the monomials in the Euclidean Steiner formula~\eqref{eq:Steiner-eucl}.
By taking a suitable moment of~\eqref{eq:Steiner-conic} we obtain the usual conic Steiner formula.

\begin{proof}[Proof sketch of Theorem~\ref{thm:gensteiner}]
In order to show the claimed equality in distribution~\eqref{eq:Steiner-conic} it suffices to show that the moments coincide. Let $f\colon \R^2_+\to \R$ be a Borel measurable function. In view of the decomposition~\eqref{eq:orthodecom} we can express the expectation of $f\big( \norm{\Pi_C(\vct{g})}^2,\norm{\Pi_{C^\polar}(\vct{g})}^2\big)$ as 
\begin{equation*}
  \Expect\big[f\big( \norm{\Pi_C(\vct{g})}^2,\norm{\Pi_{C^\polar}(\vct{g})}^2\big)\big] = \sum_{k=0}^d\sum_{F\in \F_k(C)} \Expect[f(\norm{\Proj_C(\vct{g})}^2,\norm{\Proj_{C^{\polar}}(\vct{g})}^2)\, 1_{\{\Proj_C(\vct{g})\in\relint(F)\}}].
\end{equation*}
Notice now that for $g\in (\relint F)+N_FC$ we have $\Proj_C(\vct{g}) = \Proj_{\lin(F)}(\vct{g})$ and $\Proj_{C^{\polar}}(\vct{g})=\Proj_{\lin(N_FC)}(g)$. This implies
\begin{align*}
   & \Expect\big[f\big(\norm{\Proj_C(\vct{g})}^2,\norm{\Proj_{C^{\polar}}(\vct{g})}^2\big)\, 1_{\{\Proj_C(\vct{g})\in \relint(F)\}}\big]
\\ & = \Expect\big[f\big(\norm{\Proj_{\lin(F)}(\vct{g})}^2,\norm{\Proj_{\lin(N_FC)}(\vct{g})}^2\big)\, 1_{\{\Proj(\vct{g})\in \relint(F)\}}\big].
\end{align*}
Using spherical coordinates and the orthogonal invariance of Gaussian vectors, one can deduce that the above expectation equals
\begin{align*}
  & \Expect\big[f\big(\norm{\Proj_{\lin(F)}(\vct{g})}^2,\norm{\Proj_{\lin(N_FC)}(\vct{g})}^2\big)\, 1_{\{\Proj_C(\vct{g})\in \relint(F)\}}\big]
\\ & = \Expect\big[f\big(\norm{\Proj_{L_k}(\vct{g})}^2,\norm{\Proj_{L_k^{\perp}}(\vct{g})}^2\big)\big] \, \Prob\{\Proj_C(\vct{g})\in \relint(F)\} = \Expect[f(X_k,Y_{d-k})] \, v_F(C),
\end{align*}
where $L_k$ denotes an arbitrary $k$-dimensional linear subspace. Summing up the terms gives rise to the claimed coincidence of moments, which shows equality of the distributions.
\end{proof}

A useful consequence of the general Steiner formula is that the moment generating functions of the discrete random variable $V$ from Theorem~\ref{thm:gensteiner} and the continuous random variable $\|\Pi_C(\vct g)\|^2$ coincide up to reparametrization:
  \[ \Expect[e^{tV}] = \Expect[e^{s\|\Pi_C(\vct g)\|^2}] ,\qquad s=\tfrac{1-e^{-2t}}{2} \]
which directly follows from~\eqref{eq:Steiner-conic} by the well-known formula for the moment generating function of $\chi^2$-distributed random variables, $\Expect[e^{sX_k}]=(1-2s)^{-k/2}$. This result is from~\cite{mccoy2014steiner}, where it is used to derive a concentration result for the random variable $V$,
and also underlies the argumentation in~\cite{goldstein2014gaussian}, where a central limit theorem for $V$ is derived. 

\section{Gauss-Bonnet and the face lattice}\label{sec:Gauss-Bonnet}

The Gauss-Bonnet Theorem is a celebrated result in differential geometry connecting curvature with the Euler characteristic. In the setting of polyhedral cones, this theorem asserts that the alternating sum of the intrinsic volumes equals the alternating sum of the $f$-vector,
\begin{equation*}
 \sum_{k=0}^d (-1)^k v_k(C) = \sum_{k=0}^d (-1)^k f_k(C). 
\end{equation*}

The main goal of this section is to show the connections between the Gauss-Bonnet relation, a result by Sommerville~\cite{sommerville1927relations}, which can be seen as a conic version of the Brianchon-Euler-Gram relation for polytopes~\cite[14.1]{grunbaum1967convex}, and a result by Gr\"unbaum~\cite[Thm.~2.8]{G:68}. More precisely, we will provide an elementary proof of the result by Sommerville, which is basically an application of Farkas' Lemma, and show how the other relations are easily deduced from this.
The derivation of the Gauss-Bonnet relation from the Sommerville relation presented here follows McMullen~\cite{mcmullen1975non}, who used the language of internal and external angles (see Section~\ref{sub:extangle}).

\begin{theorem}[Sommerville]\label{thm:somm}
For any polyhedral cone $C\subseteq \R^d$,
\begin{equation}\label{eq:somm}
  v_0(C) = \sum_{F\in \F(C)} (-1)^{\dim F} v_0(F) .
\end{equation}
\end{theorem}

\begin{proof}
Both sides in~\eqref{eq:somm} are zero if $C$ contains a nonzero linear subspace. So we assume in the following that $C$ is pointed, $C\cap (-C)=\vct0$. Let $\vct{g}$ be a random Gaussian vector and $H=\vct{g}^{\perp}$ the orthogonal complement, which is almost surely a hyperplane. By Farkas' Lemma~\ref{le:farkas},
\begin{equation}\label{eq:11}
 \Prob\{C\cap H = \zerovct\} = \Prob\{\vct{g}\in C^{\circ}\cup -C^{\circ}\} = 2\Prob\{\vct{g}\in C^{\circ}\}=2v_0(C).
\end{equation}
Note that with probability $1$, the intersection $C\cap H$ is either $\zerovct$ or has dimension $\dim C-1$.
Setting $\overline{\chi}=\sum_{i=0}^{d-1}(-1)^if_i(C\cap H)$, the Euler relation~\eqref{Euler} implies $\overline{\chi}=0$ if $C\cap H\neq \zerovct$ and $\overline{\chi}=1$ if $C\cap H= \zerovct$. Using~\eqref{eq:11} we get the expected value
\begin{equation}\label{eq:12}
 \Expect\left[\overline{\chi}\right] = \Expect\left[\overline{\chi} \ |\ C\cap H\neq \zerovct\right](1-2v_0(C))+
 \Expect\left[\overline{\chi}\ |\ C\cap H= \zerovct\right]2v_0(C)= 2v_0(C).
\end{equation}
On the other hand, for $0<i<d$ and using~\eqref{eq:11},
\begin{equation*}
 \Expect\left[f_i(H\cap C)\right]=\sum_{F\in \F_{i+1}(C)}\Prob\{F\cap H\neq \zerovct\} = f_{i+1}(C)-2\sum_{F\in \F_{i+1}(C)} v_0(F) ,
\end{equation*}
where in the first step we used the fact that almost surely every $i$-dimensional face of $C\cap H$ is of the form $F\cap H$, with $F\in \mF_{i+1}(C)$, and for every $F\in \mF_{i+1}(C)$ the intersection $F\cap H$ is either an $i$-dimensional face of $C\cap H$ or $\vct0$. Alternating the sum and using linearity of expectation,
\begin{align*}
 \Expect[\overline{\chi}]&=1+\sum_{i=1}^{d-1}(-1)^i \Expect[f_i(C\cap H)]= 1+\sum_{i=1}^{d-1}(-1)^i\bigg( f_{i+1}(C)-2\sum_{F\in \F_{i+1}(C)} v_0(F) \bigg)\\
 &=1-\sum_{i=2}^{d}(-1)^{i}f_{i}(C)+2\sum_{i=2}^{d}\sum_{F\in\F_i(C)} (-1)^{\dim F} v_0(F)\\
 &=1+f_0(C)-f_1(C)-\sum_{i=0}^{d}(-1)^{i}f_{i}(C)+2\bigg(-v_0(\zerovct)+\sum_{F\in\mF_1(C)}v_0(F)+\sum_{F\in \F(C)} (-1)^{\dim F} v_0(F)\bigg)\\
 &= 2\sum_{F\in \F(C)} v_0(F)(-1)^{\dim F},
\end{align*}
where in the final step we used the Euler relation~\eqref{Euler}, the fact that $f_1(C)=2\sum_{\dim F=1} v_0(F)$ (because each $F^{\circ}$ is a halfspace), and $f_0(C)=v_0(\vct{0})=1$. Combining this with~\eqref{eq:12} yields the claim.
\end{proof}

The following theorem is a simple generalization of Sommerville's Theorem. Recall from Section~\ref{sec:intr-vols} that $v_{G}(F)=0$ if $G$ is not contained in $F$.

\begin{theorem}\label{thm:2}
Let $C\subseteq\R^d$ be a polyhedral cone. Then for any face $G\subseteq C$,
\begin{equation}\label{eq:somm-gen}
  (-1)^{\dim G} v_G(C) = \sum_{F\in \F(C)} (-1)^{\dim F} v_G(F) .
\end{equation}
\end{theorem}

\begin{proof}
If $G=\vct0$ then we obtain Sommerville's Theorem~\ref{thm:somm}. Let $G\neq \zerovct$ and let $C/G$ denote the orthogonal projection of~$C$ onto the orthogonal complement of the linear span of~$G$. It follows from the Gaussian distribution that $v_G(C)=v_{G}(G)\,v_0(C/G)$, which can be expressed as
\begin{equation*}
   v_{G}(G)\,v_0(C/G) = v_G(G)\sum_{F/G\in \F(C/G)} (-1)^{\dim F/G}v_0(F/G) = \sum_{F\in \F(C)} (-1)^{\dim F-\dim G}v_G(F),
\end{equation*}
where in the first step we used Sommerville's Theorem, and in the second step we used that $v_G(F)=0$ if $G$ is not a face of $F$, and $\dim F/G=\dim F-\dim G$. This shows the claim.
\end{proof}

The following corollary is~\cite[Thm.~2.8]{G:68}, cf.~Section~\ref{sec:grassm-ang}.

\begin{corollary}\label{thm:1}
Let $C\subseteq\R^d$ be a closed convex cone. Then
\begin{equation}\label{eq:gbl}
  (-1)^k v_k(C) = \sum_{F\in \F(C)} (-1)^{\dim F} v_k(F).
\end{equation}
\end{corollary}

\begin{proof}
Follows by summing in~\eqref{eq:somm-gen} over all $k$-dimensional faces and noting that for every face $F$ of $C$ we have $\F_k(F)\subseteq\F_k(C)$.
\end{proof}

\begin{corollary}[Gauss-Bonnet]\label{cor:GB}
For a polyhedral cone $C$,
\begin{equation}\label{eq:gaussbonnet}
  \sum_{i=0}^d(-1)^i v_i(C)=\sum_{i=0}^d(-1)^i f_i(C)
    = \begin{cases}
        (-1)^{\dim C} & \text{ if } C \text{ is a linear subspace,}\\
        0 & \text{ else.}
      \end{cases}
\end{equation}
\end{corollary}

\begin{proof}
Summing the terms in~\eqref{eq:gbl} over $k$ and using $\sum_{k=0}^{d} v_k(C)=1$ yields
\begin{equation*}
  \sum_{k=0}^{d} (-1)^k v_k(C) = \sum_{k=0}^{d} \sum_{F\in \F(C)} (-1)^{\dim F} v_k(F) = \sum_{F\in \F(C)} (-1)^{\dim F} \sum_{k=0}^{d} v_k(F) = \sum_{k=0}^{d} (-1)^k f_k(C) .
\end{equation*}
The rest follows from the Euler relation~\eqref{Euler}.
\end{proof}

If $C$ is not a linear subspace, then the Gauss-Bonnet relation can be interpreted as saying that the random variable $V$ on $\{0,1,\ldots,d\}$ given by $\Prob\{V=k\}=v_k(C)$, actually decomposes into two random variables $V^0,V^1$ on $\{0,2,4,\ldots,2\lfloor\frac{d}{2}\rfloor\}$ and $\{1,3,5,\ldots,2\lfloor\frac{d-1}{2}\rfloor+1\}$, respectively, such that
  \[ \Prob\{V^0=k\}=2v_{k}(C) \quad\text{if $k$ even}, \qquad \Prob\{V^1=k\}=2v_k(C) \quad\text{if $k$ odd}. \]
In fact, the same argument that gives the general Steiner formula~\eqref{eq:Steiner-conic} also shows that
\begin{align*}
   \big( \norm{\Pi_C(\vct{g}^0)}^2,\norm{\Pi_{C^\polar}(\vct{g}^0)}^2\big) & \stackrel{d}{=} \big( X_{V^0},Y_{d-V^0}\big) , & \big( \norm{\Pi_C(\vct{g}^1)}^2,\norm{\Pi_{C^\polar}(\vct{g}^1)}^2\big) & \stackrel{d}{=} \big( X_{V^1},Y_{d-V^1}\big) ,
\end{align*}
where $\vct g^0$ and $\vct g^1$ denote Gaussian vectors conditioned on their projection on $C$ falling in an even- or odd-dimensional face, respectively, and $X_k,Y_k$ are independent $\chi^2$-distributed random variables with~$k$ degrees of freedom. We can paraphrase~\eqref{eq:gbl} in terms of the moments of these random variables.

\begin{corollary}
Let $f\colon\R_+^2\to\R$ be a Borel measurable function, and for $C\subseteq\R^d$ a polyhedral cone, which is not a linear subspace, let $\phi_f(C),\phi_f^0(C),\phi_f^1(C)$ denote the moments
\begin{align*}
   \phi_f(C) & = \Expect\big[f\big( \norm{\Pi_C(\vct{g})}^2,\norm{\Pi_{C^\polar}(\vct{g})}^2\big)\big] , & \phi_f^{0/1}(C) & = \Expect\big[f\big( \norm{\Pi_C(\vct{g}^{0/1})}^2,\norm{\Pi_{C^\polar}(\vct{g}^{0/1})}^2\big)\big] .
\end{align*}
Then we have
\begin{align*}
   \frac{\phi_f^0(C)-\phi_f^1(C)}{2} & = \sum_{F\in\mF(C)} (-1)^{\dim F} \phi_f(F) ,
 & \phi_f(C) & = \sum_{F\in\mF(C)} (-1)^{\dim F} \frac{\phi_f^0(F)-\phi_f^1(F)}{2} .
\end{align*}
\end{corollary}

\begin{proof}
The first equation is obtained by invoking the general Steiner formula and applying~\eqref{eq:gbl}:
\begin{align*}
   \frac{\phi_f^0(C)-\phi_f^1(C)}{2} & = \sum_{k=0}^d (-1)^k \Expect\big[f\big(X_k,Y_{d-k}\big)\big] v_k(C)
\\ & = \sum_{k=0}^d \Expect\big[f\big(X_k,Y_{d-k}\big)\big] \sum_{F\in \F(C)} (-1)^{\dim F} v_k(F) = \sum_{F\in \F(C)} (-1)^{\dim F} \phi_f(F) .
\end{align*}
The second equation is obtained by using M\"obius inversion~\eqref{eq:moeb-inv} and noting that the M\"obius function of the face lattice is $\mu(F,C)=(-1)^{\dim C-\dim F}$.
\end{proof}

We list a few more corollaries, the usefulness of which may yet need to be established.
The proofs are variations of the proof of Corollary~\ref{cor:GB}.

\begin{corollary}\label{cor:statdim}
For the statistical dimension $\delta(C)$ we obtain
\begin{equation*}
 \sum_{k=0}^{d} (-1)^k k \cdot v_k(C) = \sum_{F\in \F(C)} (-1)^{\dim F} \delta(F) .
\end{equation*}
In particular, if $\dim C$ is odd, then
\begin{equation*}
  2\sum_{k \text{ even}} k \ v_k(C) = \sum_{F\subset C} (-1)^{\dim F} \delta(F),
\end{equation*}
and if $\dim C$ is even, then
\begin{equation*}
  2\sum_{k \text{ odd}} k \ v_k(C) = -\sum_{F\subset C} (-1)^{\dim F} \delta(F).
\end{equation*}
\end{corollary}

\begin{corollary}\label{cor:genfun}
Let $V_C$ be the random variable on $\{0,1,\dots,d\}$ defined by $\Prob\{V_C=k\} = v_k(C)$. The alternating sum of the exponential generating function satisfies
 \begin{equation*}
  \Expect\left[(-1)^{V_C}e^{tV_C}\right] = \sum_{F\in \F(C)} (-1)^{\dim F} \Expect\left[e^{tV_F}\right] .
 \end{equation*}
\end{corollary}

\begin{remark}
The Gauss-Bonnet relation can also be written out as $\sum_{F\in\mF(C)} (-1)^{\dim F}v_F(C)=0$, if~$C$ is not a linear subspace. If $G\in\mF(C)$ is a proper face, i.e., $G\neq C$, then one can apply Gauss-Bonnet to the projected cone $C/G$, as in the deduction of Theorem~\ref{thm:2} from Sommerville's Theorem~\ref{thm:somm}, to obtain
  \[ \sum_{F\in\mF(C)} (-1)^{\dim F} v_{F/G}(C/G) = 0 . \]
Rewriting this formula in terms of internal/external angles, and extending this to include also the case $G=C$, one obtains
  \[ \sum_{G\leq F\leq C} (-1)^{\dim F-\dim G} \beta(G,F)\,\gamma(F,C) = \begin{cases} 1 & \text{if } F=G \\ 0 & \text{else}, \end{cases} \]
where $\leq$ denotes the order relation in the face lattice, i.e., the inclusion relation.
In~\cite{mcmullen1975non} McMullen observed that this relation means that the internal and external angle functions (one of them multiplied by the M\"obius function) are mutual inverses in the incidence algebra of the face lattice, cf.~Section~\ref{sec:posets}. More precisely, the Gauss-Bonnet relation only shows that one of them is the left-inverse of the other (and of course the other is a right-inverse of the first), but since left-inverse, right-inverse, or two-sided inverse are equivalent in the incidence algebra~\cite[Prop.~3.6.3]{EC} one obtains the following additional relation ``for free'':
  \[ \sum_{G\leq F\leq C} (-1)^{\dim C-\dim F} \gamma(G,F)\,\beta(F,C) = \begin{cases} 1 & \text{if } F=G \\ 0 & \text{else} . \end{cases} \]
This is~\cite[Thm.~3]{mcmullen1975non}.
\end{remark}

The relation~\eqref{eq:11} used in the proof of Sommerville's Theorem~\ref{thm:somm} is a special case of the principal kinematic formula, to be derived in more detail next.

\section{Elementary kinematics for polyhedral cones}\label{sec:kin-form}

The principal kinematic formulae of integral geometry relate the intrinsic volumes, or certain measures that localize these quantities, of the intersection of two or more randomly moved geometric objects to those of the individual objects. This section presents a direct derivation of the principal kinematic formula in the setting of two polyhedral cones. The results of this section are special cases of Glasauer's Kinematic Formula for spherically convex sets~\cite{Gl,scwe:08}, though in the spirit of the rest of this article, our proof is combinatorial, based on the facial decomposition of the cone, and uses probabilistic terminology.

In what follows, when we say that $\mtx{Q}$ is drawn uniformly at random from the orthogonal group $O(d)$, we mean that it is drawn from the Haar probability measure $\nu$ on $O(d)$. This is the unique regular Borel measure on $O(d)$ that is left and right invariant ($\nu(\vct{Q}A)=\nu(A\vct{Q})=\nu(A)$ for $\vct{Q}\in O(d)$ and a Borel measurable $A\subseteq O(d)$) and satisfies $\nu(O(d))=1$. Moreover, for measurable $f\colon O(d)\to \R_+$, we write
\begin{equation*}
 \Expect_{\vct{Q}\in O(d)}[f(\vct{Q})] := \int_{\vct{Q}\in O(d)} f(\vct{Q}) \ \nu(\diff{\mtx{Q}})
\end{equation*}
for the integral with respect to the Haar probability measure, and we will occasionally omit the subscript $\vct{Q}\in O(d)$, or just write $\mtx{Q}$ in the subscript, when there is no ambiguity. 
More information on invariant measures in the context of integral geometry can be found in~\cite[Chapter 13]{scwe:08}.

\begin{theorem}[Kinematic Formula]\label{thm:kinematic}
Let $C,D\subseteq \R^d$ be polyhedral cones.
Then, for $\vct Q\in O(d)$ uniformly at random, and $k>0$,
\begin{align}
\label{eq:kinematic}
  \Expect[v_k(C\cap \vct QD)] & = v_{k+d}(C\times D) , & \Expect[v_0(C\cap \vct QD)] & = \sum_{j=0}^d v_j(C\times D) .
\intertext{If $D=L$ is a linear subspace of dimension $d-m$, then}
\label{eq:crofton}
  \Expect[v_k(C\cap \vct QL)] & = v_{k+m}(C) , & \Expect[v_0(C\cap \vct QL)] & = \sum_{j=0}^m v_j(C) .
\end{align}
\end{theorem}

Implicit in the statement of the theorem is the integrability of $v_k(C\cap \mtx{Q}D)$ as a function of $\mtx{Q}$. This will be established in the proof.
Recall that the intrinsic volumes of $C\times D$ are obtained by convoluting the intrinsic volumes of~$C$ and~$D$, cf.~Section~\ref{sec:intr-vols}. The second equation in~\eqref{eq:kinematic} follows from the first and from~$\sum_k v_k(C)=1$, and statement~\eqref{eq:crofton} follows from~\eqref{eq:kinematic} by applying the product rule~\eqref{eq:productrule}. Note also that using polarity (Proposition~\ref{le:dualintersect}) on both sides of~\eqref{eq:kinematic} we obtain the polar kinematic formulas
\begin{align}\label{eq:kinematic-polar}
  \Expect[v_{d-k}(C+\vct QD)] & = v_{d-k}(C\times D) , & \Expect[v_d(C+\vct QD)] & = \sum_{j=0}^d v_{d+j}(C\times D) ,
\end{align}
and similarly for~\eqref{eq:crofton}. Combining Theorem~\ref{thm:kinematic} with the Gauss-Bonnet relation~\eqref{eq:gaussbonnet} yields the so-called {\em Crofton formulas}, which we formulate in the following corollary. They relate the Grassmann angles (see Section~\ref{sec:grassm-ang}) to the intrinsic volumes.

\begin{corollary}\label{cor:Croft}
Let $C,D\subseteq \R^{d}$ be polyhedral cones such that not both of $C$ and $D$ are linear subspaces. Then, for $\vct Q\in O(d)$ uniformly at random,
\begin{align*}
  \Prob\{C\cap \vct QD \neq \zerovct\} & = 2\sum_{i\geq1 \text{ odd}} v_{d+i}(C\times D).
\intertext{In particular, if $D=L$ is a linear subspace of dimension $d-m$,}
  \Prob\{C\cap \vct QL \neq \zerovct\} & = 2\sum_{i\geq1 \text{ odd}} v_{m+i}(C).
\end{align*}
\end{corollary}

For the derivation of this corollary, and for later use, we need the following genericity lemma. Recall from Section~\ref{sec:duality} that two cones $C,D\subseteq \R^d$ are said to intersect transversely, written $C\pitchfork D$, if $\relint(C)\cap \relint(D)\neq \emptyset$ and $\dim(C\cap D) = \dim(C)+\dim(D)-d$. For the rest of this section, we use the notation $L_C:=\lin(C)=C+(-C)$ for the linear span of a convex cone $C$.

\begin{lemma}\label{lem:gen-inters}
Let $C,D\subseteq \R^{d}$ polyhedral cones. Then for $\vct Q\in O(d)$ uniformly at random, almost surely either $C\cap \vct{Q}D=\vct0$ or $C\pitchfork \mtx{Q}D$ holds. In particular, if not both of $C$ and $D$ are linear subspaces, then almost surely either $C\cap \vct{Q}D=\vct0$, or $C\cap \mtx{Q}D$ is not a linear subspace.
\end{lemma}

\begin{proof}
 The 
set $\mathcal{S}$ of $\mtx{Q}\in O(d)$ with $\dim L_F\cap \mtx{Q}L_G\neq \max\{0,\dim L_F+\dim L_G-d\}$ for some $(F,G)\in \mathcal{F}(C)\times \mathcal{F}(D)$ has measure zero, see for example~\cite[Lemma 13.2.1]{scwe:08}. 

Assume $\vct Q\not\in \mathcal{S}$ and
$C\cap \mtx{Q}D\neq \zerovct$. If $\relint(C)\cap \relint(\mtx{Q}D)\neq \emptyset$, then $\mtx{Q}\not\in \mathcal{S}$
implies $\dim C\cap \mtx{Q}D=\dim C+\dim D-d$, and hence $C\pitchfork \mtx{Q}D$.
If $\relint(C)\cap \relint(\mtx{Q}D)=\emptyset$, then by the Separating Hyperplane Theorem~\ref{thm:sephyp} there exists a hyperplane $H$ such that $C\subseteq H_+$ and $\mtx{Q}D\subseteq H_-$. Let $F=C\cap H$ and $G=D\cap \mtx{Q}^{T}H$. By the assumption $C\cap \mtx{Q}D\neq \zerovct$, we have $F\neq \zerovct$ and $G\neq \zerovct$. Since $L_F$ and $\mtx{Q}L_G$ are in $H$ and $\dim H=d-1$, we get $\dim L_F\cap \mtx{Q}L_G \geq \dim L_F+\dim L_G-d+1$ and therefore $\mtx{Q}\in \mathcal{S}$, which contradicts our assumption.
We thus established 
$\{\mtx{Q}\in O(d)\mid C\cap \mtx{Q}D\neq \zerovct \text{ and } C\not\pitchfork \mtx{Q}D\}\subseteq\mathcal{S}$.

For the second claim, assume that $C$ is not a linear subspace. The lineality space of $C$, $C\cap (-C)$, is contained in every supporting hyperplane of $C$, and therefore does not intersect $\relint(C)$. If $C\pitchfork \vct QD$, then there exists nonzero $\vct x\in\relint(C)\cap \mtx{Q}D$. In particular, $\vct x$ does not lie in the lineality space of $C$.
Since the lineality space of the intersection $C\cap\vct QD$ is the intersection of the lineality spaces of $C$ and of $\vct QD$, it follows that $\vct{x}$ is in the complement of the lineality space of $C\cap\vct QD$ in $C\cap\vct QD$, which shows that $C\cap \mtx{Q}D$ is not a linear subspace. 
\end{proof}

\begin{proof}[Proof of Corollary~\ref{cor:Croft}]
Denoting $\chi(C):=\sum_{i=0}^d (-1)^i v_i(C)$, the Gauss-Bonnet relation~\eqref{eq:gaussbonnet} says that $\chi(C)=0$ if $C$ is not a linear subspace, and $\chi(\vct0)=1$. By Lemma~\ref{lem:gen-inters} we see that~$\chi$ is almost surely the indicator function for the event that $C$ and $D$ only intersect at the origin. We can therefore conclude,
\begin{align*}
  \Prob\{C\cap \vct QD = \zerovct\} &= \Expect[\chi(C\cap \vct QD)] = \Expect[\hspace{-2mm}\underbrace{v_0(C\cap \vct QD)}_{=1-\sum_{i=1}^d v_i(C\cap \vct QD)}\hspace{-2mm}]+\sum_{i=1}^d (-1)^i \Expect[v_{i}(C\cap \vct QD)]
\\ & = 1-2\sum_{i\geq1 \text{ odd}} \Expect[v_i(C\cap \vct QD)] \stackrel{~\eqref{eq:kinematic}}{=} 1-2\sum_{i\geq1 \text{ odd}} v_{d+i}(C\times D). 
\end{align*}
The second claim follows by replacing $D$ with $L$.
\end{proof}


Our proof of Theorem~\ref{thm:kinematic} is based on a classic ``double counting'' argument; to illustrate this, we first consider an analogous situation with finite sets. We note that Proposition~\ref{thm:kinematic1} generalizes without difficulties to the setting of compact groups acting on topological spaces, as in~\cite[Theorem 13.1.4]{scwe:08}.

\begin{proposition}\label{thm:kinematic1}
 Let $\Omega$ be a finite set and $G$ be a finite group acting transitively on $\Omega$. Let $M,N\subseteq \Omega$ be subsets. Then for uniformly random $\gamma\in G$,
 \begin{equation}\label{eq:kin-form-finite}
  \Expect_{\gamma\in G}|M\cap \gamma N| = \frac{|M||N|}{|\Omega|}.
 \end{equation}
\end{proposition}

\begin{proof}
Taking $\xi\in\Omega$ uniformly at random, we obtain the cardinality of $M$ as $|\Omega|\cdot\Prob\{\xi\in M\}$. Introduce the indicator function $1_M(\xi)$ for the event $\xi\in M$ and note that $1_{\gamma N}(x)=1_N(\gamma^{-1}x)$ and $\Expect_{\gamma\in G}[1_N(\gamma^{-1}x)] = |N|/|\Omega|$ for any $x\in\Omega$. It follows that the random variables $1_M(\xi),1_{\gamma N}(\xi)$ are uncorrelated:
\begin{align*}
   \Expect_{\gamma\in G}|M\cap \gamma N| & = |\Omega|\cdot\Expect_{\gamma\in G}[ \Expect_{\xi\in \Omega}[1_M(\xi)\, 1_{\gamma N}(\xi)]] = |\Omega|\cdot\Expect_{\xi\in \Omega}[ 1_M(\xi)\, \Expect_{\gamma\in G}[1_N(\gamma^{-1}\xi)]]
\\ & = \Expect_{\xi\in \Omega}[ 1_M(\xi)] \cdot |N| = \frac{|M||N|}{|\Omega|} . \qedhere
\end{align*}
\end{proof}

Lemma~\ref{lem:prekinematic} uses the same idea to establish the kinematic formula for the Gaussian measure of cones of different dimensions, and Theorem~\ref{thm:kinematic} then follows by applying this to the pairwise intersection of faces.

\begin{lemma}\label{lem:prekinematic}
 Let $C,D\subseteq \R^d$ be polyhedral cones with $\dim C=j$ and $\dim D=\ell$, and assume $0<k\leq d$. If $j+\ell=k+d$, then for $\mtx{Q}\in O(d)$ uniformly at random,
 \begin{equation}\label{eq:pre-kinematic}
  \Expect[v_k(C\cap \vct QD)] = v_j(C)\,v_\ell(D) .
\end{equation}
If $j+\ell = d-k$, then for $\mtx{Q}\in O(d)$ uniformly at random,
\begin{equation}\label{eq:pre-kinematic-dual}
   \Expect[v_{d-k}(C+ \vct QD)] = v_{j}(C)\,v_{\ell}(D).
\end{equation}
\end{lemma}

The proof of Lemma~\ref{lem:prekinematic} relies crucially on the left and right invariance of the Haar measure, which implies that for any measurable $f\colon O(d)\to \R_+$ and fixed $\vct{Q}_0,\vct{Q}_1\in O(d)$,
\begin{equation}\label{eq:orthint}
 \Expect_{\vct{Q}\in O(d)}[f(\vct{Q}\vct{Q}_0)] = \Expect_{\vct{Q}\in O(d)}[f(\vct{Q}_1\vct{Q})] = \Expect_{\vct{Q}\in O(d)}[f(\vct{Q})].
\end{equation} 
For a linear subspace $L\subseteq \R^d$, we can (and will) naturally identify the group $O(L)$ of orthogonal transformations of $L$ with the subgroup of $O(d)$ consisting of those $\vct{Q}\in O(d)$ for which $\vct{Q}\vct{x}=\vct{x}$ for $\vct{x}\in L^{\perp}$. 
The group $O(L)$ carries its own Haar probability measure. We also use the following characterization of the Gaussian volume of a convex cone $C\subseteq \R^d$:
\begin{equation}\label{eq:Qgauss}
 v_d(C) = \Expect_{\mtx{Q}\in O(d)}[1_C(\mtx{Q}\vct{x})],
\end{equation}
where $\vct{x}\neq \zerovct$ arbitrary. This characterization 
follows from the fact that for $\mtx{Q}\in O(d)$ uniformly at random, the point $\mtx{Q}\vct{x}$ is uniformly distributed on the sphere of radius $\|\vct x\|$.

\begin{proof}[Proof of Lemma~\ref{lem:prekinematic}]
For illustration purposes we first prove~\eqref{eq:pre-kinematic} in the case $k=d$, as it is almost a carbon copy of the proof of Proposition~\ref{thm:kinematic1} and~\cite[Theorem 13.1.4]{scwe:08}.
We need to show that
\begin{equation*}
  \Expect_{\vct Q\in O(d)}[v_d(C\cap \vct QD)] = v_d(C)\,v_d(D) .
\end{equation*}
Note that the map $\mtx{Q}\mapsto v_d(C\cap \vct{Q}D)$ is in fact measurable; this follows from the characterization
\begin{equation*}
 v_d(C\cap \mtx{Q}D) = \Expect_{\vct g\sim \mN(\R^d)}[1_{C\cap \mtx{Q}D}(\vct g)] = \Expect_{\vct g\sim \mN(\R^d)}[1_C(\vct g)\,1_{\vct D}(\mtx{Q}^T\vct g)],
\end{equation*}
the measurability of $(\vct{x},\mtx{Q})\mapsto 1_C(\vct{x})1_D(\mtx{Q}^T\vct{x})$, and the fact that the integral $\Expect_{\vct g\sim \mN(\R^d)}[1_C(\vct g)\,1_{D}(\vct{Q}^T\vct g)]$ is then measurable in $\mtx{Q}$, see for example~\cite[Theorem 8.5]{rudin}.
Fubini's Theorem and~\eqref{eq:Qgauss} then yield
\begin{align*}
   \Expect_{\vct Q\in O(d)}[v_d(C\cap \vct QD)] & = \Expect_{\vct Q\in O(d)}[\Expect_{\vct g\sim \mN(\R^d)}[1_C(\vct g)\,1_{\vct QD}(\vct g)]]
\\ & = \Expect_{\vct g\sim \mN(\R^d)}[1_C(\vct g)\,\Expect_{\vct Q\in O(d)}[1_D(\vct Q^T\vct g)]]
\\ & = \Expect_{\vct g\sim \mN(\R^d)}[1_C(\vct g)]\,v_d(D) = v_d(C)\,v_d(D) .
\end{align*}

We proceed with the general case of~\eqref{eq:pre-kinematic}.
%
%
By Lemma~\ref{lem:gen-inters}, almost surely $\dim L_C\cap \vct QL_D = k$ and
 $\dim(C\cap \vct QD)=k$ or $C\cap \vct QD =\zerovct$. For generic $\mtx{Q}$ we can therefore write
  \[ v_k(C\cap \vct QD)=\Expect_{\vct g\in\mN(L_C\cap \vct{Q}L_D)}[1_C(\vct g)\,1_{\vct QD}(\vct g)] =  \Expect_{\vct g\in\mN(L_C\cap \vct{Q}L_D)}[1_C(\vct g)\,1_{D}(\vct{Q}^{T}\vct g)]. \]
We thus need to show that
\begin{equation}\label{eq:weneedtoshow}
 \Expect_{\vct Q\in O(d)}\Expect_{\vct g\in\mN(L_C\cap \vct{Q}L_D)}[1_C(\vct g)\,1_{\vct D}(\vct{Q}^{T}\vct g)] = v_j(C)\,v_\ell(D).
\end{equation}
To see that the map $\mtx{Q}\mapsto v_k(C\cap \mtx{Q}D)$ is measurable, note that, using the fact that the orthogonal projection of a Gaussian vector to a subspace is again Gaussian, we have
\begin{equation*}
 \Expect_{\vct g\in\mN(L_C\cap \vct{Q}L_D)}[1_C(\vct g)\,1_{D}(\vct{Q}^{T}\vct g)] = \Expect_{\vct g\in\mN(\R^d)}[1_C(\Proj_{L_C\cap \vct{Q}L_D}(\vct g))\,1_{D}(\vct{Q}^{T}\Proj_{L_C\cap \vct{Q}L_D}(\vct g))].
\end{equation*}
It is enough to verify that the projection $\Proj_{L_C\cap \mtx{Q}L_D}(\vct{x})$ is continuous in $\vct{x}$ and $\mtx{Q}$ outside a set of measure zero; the measurability of $v_k(C\cap \mtx{Q}D)$ then follows from the same considerations as in the case $k=d$.
If $C\pitchfork \mtx{Q}D$, then $L_C\cap \mtx{Q}L_D$ is the kernel of a matrix of rank $d-k$ whose rows depend continuously on $\mtx{Q}$. The projection $\Proj_{L_C\cap \mtx{Q}L_D}(\vct{x})$ depends continuously on $\vct{x}$ and on this matrix, and therefore also on $\mtx{Q}$.  

We now proceed to show identity~\eqref{eq:weneedtoshow}.
Let $\vct{Q}_0\in O(L_D)$. By the orthogonal invariance~\eqref{eq:orthint},
\begin{align*}
 \Expect_{\vct Q\in O(d)}\Expect_{\vct g\in\mN(L_C\cap \vct{Q}L_D)}[1_C(\vct g)\,1_{D}(\vct{Q}^T\vct g)] 
 &=\Expect_{\vct Q\in O(d)}\Expect_{\vct g\in\mN(L_C\cap \vct{Q}\vct{Q}_0L_D)}[1_C(\vct g)\,1_{D}(\vct{Q}_0^{T}\vct{Q}^{T}\vct g)] .
\end{align*}
Since this holds for any $\vct{Q}_0\in O(L_D)$, we can choose $\vct{Q}_0\in O(L_D)$ uniformly at random to obtain
\begin{align*}
 \Expect_{\vct Q\in O(d)}\Expect_{\vct g\in\mN(L_C\cap \vct{Q}L_D)}[1_C(\vct g)\,1_{\vct QD}(\vct g)] 
 &=\Expect_{\vct{Q}_0\in O(L_D)}\Expect_{\vct Q\in O(d)}\Expect_{\vct g\in\mN(L_C\cap \vct{Q}\vct{Q}_0L_D)}[1_C(\vct g)\,1_{D}(\vct{Q}_0^{T}\vct{Q}^{T}\vct g)]\\
 &\stackrel{(1)}{=}\Expect_{\vct{Q}_0\in O(L_D)}\Expect_{\vct Q\in O(d)}\Expect_{\vct g\in\mN(L_C\cap \vct{Q}L_D)}[1_C(\vct g)\,1_{D}(\vct{Q}_0^{T}\vct{Q}^{T}\vct g)]\\
 & \stackrel{(2)}{=}\Expect_{\vct Q\in O(d)}\Expect_{\vct g\in\mN(L_C\cap \vct{Q}L_D)}[1_C(\vct g)\Expect_{\vct{Q}_0\in O(L_D)}[1_{D}(\vct{Q}_0^{T}\vct{Q}^{T}\vct g)]]\\
 &\stackrel{(3)}{=}\Expect_{\vct Q\in O(d)}\Expect_{\vct g\in\mN(L_C\cap \vct{Q}L_D)}[1_C(\vct g)]\ v_{\ell}(D),
\end{align*}
where in (1) we used $\vct{Q}_0L_D=L_D$, in (2) we used Fubini's Theorem, and in (3) we used~\eqref{eq:Qgauss}. For the remaining part, replacing $\vct{Q}$ with $\vct{Q}_1\vct{Q}$ for $\vct{Q}_1\in O(L_C)$
uniformly at random,
and applying~\eqref{eq:orthint} again,
\begin{align*}
 \Expect_{\vct Q\in O(d)}\Expect_{\vct g\in\mN(L_C\cap \vct{Q}L_D)}[1_C(\vct g)]
 &= \Expect_{\vct{Q}_1\in O(L_C)}\Expect_{\vct Q\in O(d)}\Expect_{\vct g\in\mN(\vct{Q}_1(L_C\cap \vct{Q}L_D))}[1_C(\vct g)]\\
 &= \Expect_{\vct{Q}_1\in O(L_C)}\Expect_{\vct Q\in O(d)}\Expect_{\vct g\in\mN(L_C\cap \vct{Q}L_D)}[1_C(\vct{Q}_1\vct g)]\\
 &= \Expect_{\vct Q\in O(d)}\Expect_{\vct g\in\mN(L_C\cap \vct{Q}L_D)}\Expect_{\vct{Q}_1\in O(L_C)}[1_C(\vct{Q}_1\vct g)]\\
 &= v_j(C),
\end{align*}
where the last equality follows again from~\eqref{eq:Qgauss}.

We now derive~\eqref{eq:pre-kinematic-dual}. By Lemma~\ref{lem:gen-inters}, for generic $\mtx{Q}$, $L_C\cap \mtx{Q}L_D=\zerovct$ and $\dim L_C+\mtx{Q}L_D = j+\ell = d-k$. Using the fact that an orthogonal projection of a Gaussian vector is Gaussian, we get
\begin{align*}
 v_{d-k}(C+\mtx{Q}D) = \Expect_{\mtx{g}\in \mN(L_C+\mtx{Q}L_D)}[1_{C+\mtx{Q}D}(\mtx{g})]
 &= \Expect_{\mtx{g}\in \mN(\R^d)}[1_{C+\mtx{Q}D+(L_C+\mtx{Q}L_D)^{\bot}}(\mtx{g})]
 \\ &= \Expect_{\mtx{g}\in \mN(\R^d)}[1_{C+\mtx{Q}D+(L_C^{\bot}\cap \mtx{Q}L_D^{\bot})}(\mtx{g})].
\end{align*}
The integrability of this expression in $\mtx{Q}$ follows, as above, from the fact that the projection map to $L_C+\mtx{Q}L_D$ is continuous for almost all $\mtx{Q}$ and $\vct{g}$. For generic $\mtx{Q}$, any $\vct{g}\in \R^d$ has a unique  decomposition $\vct{g}=\vct{g}_C+\vct{g}_D+\vct{g}^{\bot}$, with $\vct{g}_C\in L_C$, $\vct{g}_D\in \mtx{Q}L_D$, $\vct{g}^{\bot}\in (L_C+\mtx{Q}L_D)^{\bot}$. Note that $\vct{g}_C$ and $\vct{g}_D$ are \emph{not} orthogonal projections, and that the decomposition (even $\vct{g}_C$) depends on $\mtx{Q}$.

From the uniqueness of this decomposition we get the equivalence
\begin{equation*}
 \vct{g}\in C+\mtx{Q}D+(L_C^{\bot}\cap\mtx{Q}L_D^{\bot}) \Longleftrightarrow \vct{g}_C\in C \text{ and } \ \vct{g}_D\in \mtx{Q}D,
\end{equation*}
and therefore 
\begin{equation*}
 \Expect[v_{d-k}(C+\mtx{Q}D)] = \Expect_{\mtx{Q}}\Expect_{\mtx{g}\in \mN(\R^d)}[1_{C+\mtx{Q}D+(L_C^{\bot}\cap\mtx{Q}L_D^{\bot})}(\mtx{g})] 
 = \Expect_{\mtx{Q}}\Expect_{\mtx{g}\in \mN(\R^d)}[1_{C}(\vct{g}_C)1_{\mtx{Q}D}(\vct{g}_D)].
\end{equation*}
Now let $\mtx{Q}_0\in O(L_C)$ be fixed. 
By orthogonal invariance of the Haar measure and of the Gaussian distribution we can replace $\mtx{Q}$ with $\mtx{Q}_0\mtx{Q}$ and $\vct{g}$ with $\vct{g}':=\mtx{Q}_0\vct{g}$. We next determine the decomposition $\vct{g}'=\vct{g}'_C+\vct{g}'_D+\vct{g}'^{\bot}$ in $L_C+\mtx{Q}_0\mtx{Q}L_D+(L_C^{\bot}\cap \mtx{Q}_0\mtx{Q}L_D^{\bot})$. Note that under this substitution,
\begin{equation*}
 \vct{g}' = \mtx{Q}_0\vct{g} = \mtx{Q}_0\vct{g}_C+\mtx{Q}_0\vct{g}_D+\mtx{Q}_0\vct{g}^{\bot},
\end{equation*}
with $\mtx{Q}_0\vct{g}_C\in L_C$ (by the fact that $\mtx{Q}_0\in O(L_C)$), $\mtx{Q}_0\vct{g}_D\in \mtx{Q}_0\mtx{Q}L_D$ (by definition), and $\mtx{Q}_0\vct{g}^{\bot}\in \mtx{Q}_0(L_C^{\bot}\cap \mtx{Q}L_D^{\bot}) = (L_C^{\bot}\cap \mtx{Q}_0\mtx{Q}L_D^{\bot})$ (since $\mtx{Q}_0$ is the identity on $L_C^{\bot}$). By uniqueness of the decomposition,
\begin{equation*}
 \vct{g}'_C = \mtx{Q}_0\vct{g}_C, \ \vct{g}'_D = \mtx{Q}_0\vct{g}_D, \ \vct{g}'^{\bot} = \mtx{Q}_0\vct{g}^{\bot} = \vct{g}^{\bot}.
\end{equation*}
We therefore have
\begin{align*}
 \Expect_{\mtx{Q}}\Expect_{\mtx{g}\in \mN(\R^d)}[1_{C}(\vct{g}_C)1_{\mtx{Q}D}(\vct{g}_D)]&= \Expect_{\mtx{Q}_0\in O(L_C)}\Expect_{\mtx{Q}}\Expect_{\vct{g}\in \mN(\R^d)}[1_{C}(\mtx{Q}_0\vct{g}_C) 1_{\mtx{Q}_0\mtx{Q}D}(\mtx{Q}_0\vct{g}_D)]\\
 &=\Expect_{\mtx{Q}}\Expect_{\vct{g}\in \mN(\R^d)}[\Expect_{\mtx{Q}_0\in O(L_C)} [1_{C}(\mtx{Q}_0\vct{g}_C)] 1_{D}(\mtx{Q}^T\vct{g}_D)]\\
 &= v_j(C) \Expect_{\mtx{Q}}\Expect_{\vct{g}\in \mN(\R^d)}[1_{D}(\mtx{Q}^T\vct{g}_D)],
\end{align*}
where we used Fubini in the second and~\eqref{eq:Qgauss} in the last equality.
Note that $\mtx{Q}^T\vct{g}_D\in L_D$. Repeating the argument above by replacing $\vct{Q}$ with $\mtx{Q}\mtx{Q}_1^T$ for $\mtx{Q}_1\in O(L_D)$, we get
\begin{align*}
 \Expect_{\mtx{Q}}\Expect_{\vct{g}\in \mN(\R^d)}[1_{D}(\mtx{Q}^T\vct{g}_D)] &= \Expect_{\mtx{Q}_1\in O(L_D)}\Expect_{\mtx{Q}}\Expect_{\vct{g}\in \mN(\R^d)}[1_{D}(\mtx{Q}_1\mtx{Q}^T\vct{g}_D)]
 \\&=\Expect_{\mtx{Q}}\Expect_{\vct{g}\in \mN(\R^d)}\Expect_{\mtx{Q}_1\in O(L_D)}[1_{D}(\mtx{Q}_1\mtx{Q}^T\vct{g}_D)]=v_{\ell}(D),
\end{align*}
where again we used~\eqref{eq:Qgauss}. This finishes the proof.
\end{proof}
%

\begin{proof}[Proof of Theorem~\ref{thm:kinematic}]
We first note that it suffices to prove the first equality in~\eqref{eq:kinematic}, as we can deduce the second from the fact that the intrinsic volumes sum up to one,
\begin{align}\label{eq:joker}
   \Expect[v_0(C\cap \vct QD)] & = \Expect\Big[1-\sum_{k>0} v_k(C\cap \vct QD)\Big] = 1-\sum_{k>0} \Expect[v_k(C\cap \vct QD)]
\\ & = \Big(\sum_i v_i(C\times D)\Big)-\Big(\sum_{k>0} v_{k+d}(C\times D)\Big) = \sum_{j=0}^d v_j(C\times D) .
\nonumber
\end{align}
The equations in~\eqref{eq:crofton} follow directly from~\eqref{eq:kinematic} as a special case, since 
\begin{equation*}
 v_{k+d}(C\times L)=\sum_{i+j=k+d} v_i(C)v_j(L)=v_{k+m}(C)v_{d-m}(L)=v_{k+m}(C)
\end{equation*}
if $L$ is a linear subspace of dimension~$d-m$.


The genericity Lemma~\ref{lem:gen-inters} implies that the $k$-dimensional faces of $C\cap\vct QD$ are generically of the form $F\cap\vct QG$ with $(F,G)\in\F(C)\times\F(D)$ and $\dim F+\dim G=k+d$. If we have shown that for $F\in \F_j(C)$, $G\in \F_{\ell}(D)$, with $j+\ell>d$, one has
\begin{equation}\label{eq:pre-crofton}
  \Expect[v_{F\cap \vct QG}(C\cap \vct QD)\cdot 1_{\{F\pitchfork \vct QG\}}] = v_F(C)\ v_G(D) ,
\end{equation}
then the kinematic formula follows by noting that $v_F(C)v_G(D)=v_{F\times G}(C\times D)$ and
\begin{align*}
 \Expect[v_k(C\cap \vct QD)] &= \sum_{\substack{(F,G)\in \F(C)\times \F(D)\\\dim F+\dim G=k+d}}\Expect[v_{F\cap \vct QG}(C\cap \vct QD)\cdot 1_{\{F\pitchfork \vct QG\}}]
\\ & = \sum_{\substack{(F,G)\in \F(C)\times \F(D)\\\dim F+\dim G=k+d}}v_{F\times G}(C\times D)=v_{k+d}(C\times D).
\end{align*}
It remains to show~\eqref{eq:pre-crofton}.
%
%
By~\eqref{eq:intvolintersect} and Lemma~\ref{lem:gen-inters}, almost surely
  \[ v_{F\cap \vct QG}(C\cap \vct QD)\cdot 1_{\{F\pitchfork \vct QG\}} = v_k(F\cap \vct QG)\,v_{d-k}(N_FC + \vct QN_GD) . \]
The integrability of these terms has been shown in the proof of Lemma~\ref{lem:prekinematic}, which shows the integrability in~\eqref{eq:kinematic}.
In order to prove~\eqref{eq:pre-crofton} we proceed as in the proof of Lemma~\ref{lem:prekinematic}. 
Let $\vct Q_0\in O(L_F)$ be uniformly at random. Note that the normal cone $N_FC$ lies in the orthogonal complement of $L_F$, so that $\vct Q_0$ leaves the normal cone invariant.
Using the invariance of the Haar measure as in the proof of Lemma~\ref{lem:prekinematic},
\begin{align*}
  \Expect_{\vct Q}[ v_{F\cap \vct QG}(C\cap \vct QD) & \cdot 1_{\{F\pitchfork \vct QG\}}] = \Expect_{\vct Q}[v_k(F\cap \vct QG)\,v_{d-k}(N_FC + \vct QN_GD)]
\\ & = \Expect_{\vct Q_0}\Expect_{\vct Q}[v_k(F\cap \vct Q_0\vct QG)\,v_{d-k}(N_FC + \vct Q_0\vct QN_GD)]
\\ & \stackrel{(1)}{=} \Expect_{\vct Q_0}\Expect_{\vct Q}[v_k(\vct Q_0^T F\cap \vct QG)\,v_{d-k}(\vct Q_0^T N_FC + \vct QN_GD)]
\\ & = \Expect_{\vct Q}\big[\Expect_{\vct Q_0}[v_k(\vct Q_0^T F\cap \vct QG)] v_{d-k}(N_FC + \vct QN_GD) \big]
\\ & = \Expect_{\vct Q}\big[\Expect_{\vct Q_0}[v_k(\vct Q_0^T F\cap (\vct QG\cap L_F))] v_{d-k}(N_FC + \vct QN_GD) \big]
\\ & = \Expect_{\vct Q}\big[\Expect_{\vct Q_0}[v_k(F\cap \mtx{Q}_0(\vct QG\cap L_F))] v_{d-k}(N_FC + \vct QN_GD) \big]
\\ & \stackrel{(2)}{=} v_j(F)\,\Expect_{\vct Q}\big[ v_k(L_F\cap\vct QG) v_{d-k}(N_FC + \vct QN_GD)\big] ,
\end{align*}
where in (1) we used the orthogonal invariance of the intrinsic volumes and in (2) we applied Lemma~\ref{lem:prekinematic} to the inner expectation (note that the dimensions match).
Comparing the first line with the last line we see that the term $v_j(F)$ could be extracted by replacing $F$ with $L_F$. Repeating the same trick by replacing $\mtx{Q}$ with $\mtx{Q}\mtx{Q}_1$ for $\mtx{Q}_1\in O(L_D)$, we get
\begin{align*}
  \Expect_{\vct Q}[ v_{F\cap \vct QG}(C\cap \vct QD) \cdot 1_{\{F\pitchfork \vct QG\}}] &= v_j(F)v_\ell(G)\,\Expect_{\vct Q}\big[ v_k(L_F\cap\vct QL_G) v_{d-k}(N_FC + \vct QN_GD)\big]\\
  &=v_j(F)v_\ell(G)\,\Expect_{\vct Q}\big[v_{d-k}(N_FC + \vct QN_GD)\big]\\
  &= v_j(F)v_\ell(G)v_{d-j}(N_FC)v_{d-\ell}(N_GD) = v_F(C)v_G(D),
\end{align*}
where in the second equation we used that $v_k(L_F\cap \mtx{Q}L_G) = 1$, and the last equality follows from~\eqref{eq:pre-kinematic-dual}.
\end{proof}

\begin{remark}
 In the literature there are roughly two different strategies used to derive kinematic formulas:
\begin{enumerate}
  \item Use a characterisation theorem for the intrinsic volumes (or a suitable localisation thereof) that shows that certain types of functions in a cone must be linear combinations of the intrinsic volumes. 
  This approach is common in integral geometry~\cite{scwe:08,KR:97}, see~\cite{Gl,A:12} for the spherical/conic setting. 
  \item Assume that the boundary of the cone intersected with a sphere is a smooth hypersurface; then argue over the curvature of the intersection of the boundaries. For a general version of this approach, with references to related work, see~\cite{howa:93}.
\end{enumerate}
The second approach is usually also based on a double-counting argument that involves the co-area formula. Our proof can be interpreted as a piecewise-linear version of this approach.
\end{remark}

\section{The Klivans-Swartz relation for hyperplane arrangements}\label{sec:Kl-Sw}

While the most natural lattice structure associated to a polyhedral cone is arguably its face lattice, there is also the intersection lattice generated by the hyperplanes that are spanned by the facets of the cone (assuming that the cone has nonempty interior; otherwise one can argue within the linear span of the cone). In this section we present a deep and useful relation between this intersection lattice and the intrinsic volumes of the regions of the hyperplane arrangement, which is due to Klivans and Swartz~\cite{klivans2011projection}, and which we will generalize to also include the faces of the regions. We finish this section with some applications of this result.

Let $\A$ be a hyperplane arrangement in~$\R^d$. Recall from~\eqref{eq:def-R_k(A)} the notation $\mR_j(\A)$ and $r_j(\A)$ for the set of $j$-dimensional regions of the arrangement and for their cardinality, respectively. Also recall Zaslavsky's Theorem~\ref{thm:zas1}, which is the briefly stated identity $r_j(\A) = (-1)^j\,\chi_{\A,j}(-1)$, where $\chi_{\A,j}$ denotes the $j$th-level characteristic polynomial of the arrangement. Expressing this polynomial in the form
  \[ \chi_{\A,j}(t) = \sum_{k=0}^j a_{jk} t^k , \]
and using the identity $\sum_k v_k(C)=1$, we can rewrite Zaslavsky's result in the form
  \[ \sum_{k=0}^j \sum_{F\in \mR_j(\A)} v_k(F) = \sum_{k=0}^j (-1)^{j-k} a_{jk} . \]
Klivans and Swartz~\cite{klivans2011projection} have proved that in the case $j=d$ this equality of sums is in fact an equality of the summands. We will extend this and show that for all~$j$ the summands are equal. In particular, taking the sum of intrinsic volumes of all regions of a certain dimension~$j$ in a hyperplane arrangement yields a quantity that is solely expressible in the lattice structure of the hyperplane arrangement. So while the intrinsic volumes of a single region are certainly not necessarily invariant under any nonsingular linear transformations, the sum of intrinsic volumes over all regions of a fixed dimension is indeed invariant under any nonsingular linear transformations.

\begin{theorem}\label{thm:klivans}
Let $\A$ be a hyperplane arrangement in $\R^d$. Then for $0\leq j\leq d$,
\begin{equation*}
  \sum_{F\in \mR_j(\A)} P_F(t) = (-1)^j\chi_{\A,j}(-t) ,
\end{equation*}
where $P_F(t) = \sum_k v_k(F)t^k$. In terms of the intrinsic volumes, for $0\leq k\leq j$,
\begin{equation}\label{eq:KS-v_k(C)}
  \sum_{F\in \mR_j(\A)} v_k(F) = (-1)^{j-k} a_{jk},
\end{equation}
where $a_{jk}$ is the coefficient of $t^k$ in $\chi_{\A,j}(t)$.
\end{theorem}

Note that in the special case $j=k$ we obtain $\sum_{F\in \mR_j(\A)} v_j(F) = \ell_j(\A)$, which is easily verified directly.
We derive a concise proof of Theorem~\ref{thm:klivans} by combining Zaslavsky's Theorem with the kinematic formula. A similar, though slightly different, proof strategy using the kinematic formula was recently employed in~\cite{KVZ:15} to derive Klivans and Swartz's result.

The cases $j=0,1$ will be shown directly; in the case $j\geq 2$ we prove~\eqref{eq:KS-v_k(C)} by induction on~$k$. This proof by induction naturally consists of two steps:
\begin{enumerate}
  \item For the case $k=0$ we need to show
  \begin{equation*}
    \sum_{F\in \mR_j(\A)} v_0(F) = (-1)^j a_{j0}.
  \end{equation*}
Let $H$ be a hyperplane in general position relative to~$\A$, that is, $H$ intersects all subspaces in $\mL(\A)$ transversely. In $H$ consider the restriction $\A^H=\{H'\cap H\mid H'\in \A\}$. The number of $(j-1)$-dimensional regions in $\A^H$ is given by the number of $j$-dimensional regions in~$\A$, which are hit by the hyperplane~$H$. With the simplest case of the Crofton formula~\eqref{eq:11}, we obtain for a uniformly random hyperplane $H$,
  \begin{equation*}
     \Expect\big[r_{j-1}(\A^H)\big] = \sum_{F\in \mR_j(\A)} \Prob\{F\cap H \neq \zerovct\} = \sum_{F\in \mR_j(\A)} (1-2v_0(F)) = r_j(\A) - 2\sum_{F\in \mR_j(\A)} v_0(F),
  \end{equation*}
  and therefore,
  \begin{equation}\label{eq:E[r(A^H)]=...}
   \sum_{F\in \mR_j(\A)} v_0(F)  = \frac{1}{2}\left(r_j(\A) - \Expect\big[r_{j-1}(\A^H)\big]\right) .
  \end{equation}
         We will see below that $r_{j-1}(\A^H)$ is almost surely constant (which eliminates the expectation on the left-hand side) and is in fact expressible in terms of $\chi_{\A,j}$. This will give the basis step in a proof by induction on~$k$ of~\eqref{eq:KS-v_k(C)}.
  \item For the induction step we use the kinematic formula~\eqref{eq:crofton} with $m=1$, that gives for a uniformly random hyperplane~$H$,
  \begin{align}\label{eq:red-kinform-KS-1}
     \sum_{F\in \mR_j(\A)} v_1(F) & = \sum_{F\in \mR_j(\A)} \big(\Expect[v_0(F\cap H)]-v_0(F)\big)
  \\\nonumber
     & = \Expect\Big[\sum_{F\in \mR_j(\A)} v_0(F\cap H)\Big] - \sum_{F\in \mR_j(\A)} v_0(F) ,
  \\\label{eq:red-kinform-KS-2}
  \sum_{F\in \mR_j(\A)} v_k(F) & = \sum_{F\in \mR_j(\A)} \Expect[v_{k-1}(F\cap H)] = \Expect\Big[\sum_{F\in \mR_j(\A)} v_{k-1}(F\cap H)\Big] , \quad \text{if } k\geq2 .
  \end{align}
        Notice that if the summation would be over the regions in $\A^H$, then we could (and in fact can if $k\geq2$) apply the induction hypothesis and express $\sum v_k(C\cap H)$ in terms of the characteristic polynomials of~$\A^H$, which, as we will see below, is constant for generic~$H$ and expressible in the characteristic polynomial of~$\A$. Since the summation is over the regions of~$\A$ we need to be a bit careful in the case $k=1$.
\end{enumerate}
To implement this idea we need to understand how the characteristic polynomial of a hyperplane arrangement changes when adding a hyperplane in general position.

\begin{lemma}\label{le:generic}
Let $\A$ be a hyperplane arrangement in~$\R^d$, and let $j\geq2$. If $H\subset\R^d$ is a linear hyperplane in general position relative to $\A$, then the $(j-1)$th-level characteristic polynomial of the reduced arrangement~$\A^H$ and the number of $(j-1)$-dimensional regions of $\A^H$ are given by
\begin{align*}
   \chi_{\A^H,j-1}(t) & = \chi_{\A,j}(0)+\frac{\chi_{\A,j}(t)-\chi_{\A,j}(0)}{t} , & r_{j-1}(\A^H) & = r_j(\A) - (-1)^j 2\chi_{\A,j}(0) .
\end{align*}
In terms of coefficients, if $\chi_{\A,j}(t)=\sum_k a_{jk}\, t^k$, then
\begin{align}\label{eq:chi_(A^L)(x)-gen}
  \chi_{\A^H,j-1}(t) & = a_{j0} + \sum_{k=1}^j a_{jk}\,t^{k-1},
  & r_{j-1}(\A^H) & = r_j(\A) - (-1)^j 2a_{j0} .
\end{align}
\end{lemma}

\begin{proof}
Note first that if $\tilde L,L\in \mL(\A)$, with $\dim \tilde L,\dim L\geq 2$, then $\tilde L\supseteq L$ if and only if $\tilde L\cap H\supseteq L\cap H$.
 Indeed, if $\tilde L\cap H\supseteq L\cap H$, then
  \[ \dim(\tilde L\cap L)-1 = \dim(\tilde L\cap L\cap H) = \dim(L\cap H)=\dim L-1\geq 1 , \]
where we used the assumption that $H$ intersects all subspaces in $\mL(\A)$ transversely. Hence, $\dim L=\dim(\tilde L\cap L)$, and $\tilde L\supseteq L$. 
In other words, the map $L\mapsto L\cap H$ is a bijection between $\mL_j(\A)$ and $\mL_{j-1}(\A^H)$ for all $j\geq 2$ that is compatible with the partial orders on $\mL(\A)$ and $\mL(\A^H)$. Of course, all elements in $\mL_0(\A)\cup\mL_1(\A)$ are mapped to~$\vct0$.

Now, recall the form of the $j$th-level characteristic polynomial~\eqref{eq:def-chi_(A,k)(t)}
\begin{align*}
   \chi_{\A,j}(t) & = \sum_{k=0}^j a_{jk} t^k , & a_{jk} & = \sum_{\tilde L\in \mL_j(\A)}\,\sum_{L\in \mL_k(\A)}\mu(\tilde L,L) ,
\end{align*}
and also recall the recursive definition of the M\"obius function~\eqref{eq:def-moeb}, $\mu(\tilde L,L)=0$ if $\tilde L\not\supseteq L$, and
\begin{align*}
   \mu(L,L) & = 1 , & \mu(\tilde L,L) & = -\sum_{\tilde L\supseteq M\supset L} \mu(\tilde L,M) \quad \text{if } \tilde L\supset L .
\end{align*}
From the above observation about the sets $\mL_j(\A)$ and $\mL_{j-1}(\A^H)$ for $j\geq 2$ we obtain
  \[ \forall \tilde L,L\in\mL(\A),\, \dim\tilde L,\dim L\geq 2: \mu(\tilde L,L) = \bar\mu(\tilde L\cap H,L\cap H) , \]
where $\bar\mu$ shall denote the M\"obius function on $\mL(\A^H)$. This shows the claimed formula for the nonconstant coefficients of $\chi_{\A^H,j-1}$. We obtain the claim for the constant coefficient by noting that for $L\in\mL(\A)$, $\dim L\geq 2$, and $\bar L:=L\cap H$,
\begin{align*}
   \bar\mu(\bar L,\vct0) & = -\sum_{\bar L\supseteq \bar M\supset \vct0} \bar\mu(\bar L,\bar M) = -\sum_{\substack{L\supseteq M\\ \dim M\geq2}} \bar\mu(L\cap H,M\cap H) = -\sum_{\substack{L\supseteq M\\ \dim M\geq2}} \mu(L,M) ,
\end{align*}
so that the constant coefficient of $\chi_{\A^H,j-1}$ is given by
\begin{align*}
   \bar a_{j-1,0} & = -\sum_{L\in \mL_k(\A)}\,\sum_{\substack{L\supseteq M\\ \dim M\geq2}} \mu(L,M) .
\end{align*}
The constant and linear coefficients of $\chi_{\A,j}$ are given by
\begin{align*}
   a_{j0} & = \sum_{L\in \mL_j(\A)}\,\mu(L,\vct0) = -\sum_{L\in \mL_j(\A)}\,\sum_{\substack{L\supseteq M\\ \dim M\geq1}} \mu(L,M)
\\ & = -\sum_{L\in \mL_j(\A)}\,\sum_{\substack{L\supseteq M\\ \dim M\geq2}} \mu(L,M) - \sum_{L\in L_j(\A)}\,\sum_{M\in \mL_1(A), L\supseteq M} \mu(L,M),
\\ a_{j1} & = \sum_{L\in \mL_j(\A)}\,\sum_{M\in \mL_1(\A), L\supseteq M}\mu(L,M) ,
\end{align*}
which shows that indeed $\bar a_{j-1,0}=a_{j0}+a_{j1}$. As for the claimed formula for $r_{j-1}(\A^H)$ we use Zaslavsky's Theorem~\ref{thm:zas1} to obtain
\begin{align*}
   r_{j-1}(\A^H) & = (-1)^{j-1}\,\chi_{\A^H,j-1}(-1) = (-1)^{j-1}\,\big( 2\chi_{\A,j}(0)-\chi_{\A,j}(-1)\big) = r_j(\A) - (-1)^j 2\chi_{\A,j}(0) ,
\end{align*}
which finishes the proof.
\end{proof}

\begin{proof}[Proof of Theorem~\ref{thm:klivans}]
We first verify the cases $j=0,1$ directly. Recall from~\eqref{eq:chi_(A,0/1)(t)} that $\chi_{\A,0}(t) = \ell_0(\A)$ and $\chi_{\A,1}(t) = \ell_1(\A) (t-\ell_0(\A))$, where $\ell_j(\A)=|\mL_j(\A)|$. In a linear hyperplane arrangement we have at most one $0$-dimensional region, and $\mR_0(\A)=\mL_0(\A)$ (possibly both empty). Therefore,
\begin{align*}
   \sum_{F\in \mR_0(\A)} P_F(t) & = r_0(\A) = \ell_0(\A) = \chi_{\A,0}(-t) .
\end{align*}
As for the case $j=1$, note first that if $r_0(\A)=0$, then $\mR_1(\A)=\mL_1(\A)$ and the claim follows as in the case $j=0$. If on the other hand $r_0(\A)=1$, then every line $L\in \mL_1(\A)$ corresponds to two rays $F_+,F_-\in \mR_1(\A)$, that is, $r_1(\A)=2\ell_1(\A)$. Since $v_1(F_\pm)=v_0(F_\pm)=\frac{1}{2}$, and $\ell_0(\A)=1$, we obtain
\begin{align*}
   \sum_{F\in \mR_1(\A)} P_F(t) & = \frac{r_1(\A)}{2} (t+1) = \ell_1(\A) (t+\ell_0(\A)) = -\chi_{\A,1}(-t) .
\end{align*}

We now assume $j\geq 2$ and proceed by induction on $k$ starting with $k=0$. In~\eqref{eq:E[r(A^H)]=...} we have seen that
  \[ \sum_{F\in \mR_j(\A)} v_0(F)  = \frac{1}{2}\left(r_j(\A) - \Expect\big[r_{j-1}(\A^H)\big]\right) . \]
From Lemma~\ref{le:generic} we obtain that~$r_{j-1}(\A^H)$ is almost surely constant and given by $r_j(\A) - (-1)^j 2\chi_{\A,j}(0)$. Therefore,
\begin{align*}
   \sum_{F\in \mR_j(\A)} v_0(F) & = \tfrac{1}{2}\Big( r_j(\A)-\big(r_j(\A) - (-1)^j 2\chi_{\A,j}(0)\big)\Big) = (-1)^j \chi_{\A,j}(0) = (-1)^j a_{j0} .
\end{align*}
This settles the case $k=0$. For $k>0$ we need to distinguish between $k=1$ and $k\geq2$. From~\eqref{eq:red-kinform-KS-1}, we obtain, using the case $k=0$ and Lemma~\ref{le:generic},
\begin{align*}
   \sum_{F\in \mR_j(\A)} v_1(F) & = \Expect\Big[\sum_{F\in \mR_j(\A)} v_0(F\cap H)\Big]-\sum_{F\in \mR_j(\A)}v_0(F)
\\ & = \Expect\Big[\sum_{\bar F\in \mR_{j-1}(\A^H)} v_0(\bar F) + |\{F\in \mR_j(\A)\mid F\cap H=\vct0\}| \Big]-(-1)^j a_{j0}
\\ & = (-1)^{j-1}(a_{j0}+a_{j1}) + \sum_{F\in \mR_j(\A)} \underbrace{\Prob\{F\cap H=\vct0\}}_{=2v_0(F)} - (-1)^j a_{j0}
\\ & = (-1)^{j-1}(a_0+a_1) + 2 (-1)^j a_{j0} - (-1)^j a_{j0} = (-1)^{j-1} a_1 .
\end{align*}
This settles the case $k=1$. Finally, in the case $k\geq2$ we argue similarly, using that $v_i(\vct0)=0$ if $i>0$,
\begin{align*}
   \sum_{F\in \mR_j(\A)} v_k(C) & \stackrel{\eqref{eq:red-kinform-KS-2}}{=} \Expect\Big[\sum_{F\in \mR_j(\A)} v_{k-1}(F\cap H)\Big] = \Expect\Big[\sum_{\bar F\in \mR_{j-1}(\A^H)} v_{k-1}(\bar F)\Big]
\\ & = (-1)^{j-1-(k-1)} \bar a_{jk} = (-1)^{j-k} a_{jk} .
\end{align*}
\end{proof}

\begin{remark}
 It was pointed out to us by Rolf Schneider that for $k>0$, $j>0$ and a subspace $L$ of dimension $\dim L=d-m$, in general position relative to $\A$, one can (as we did in the case $k=0$) use the identity
\begin{equation*}
  r_{j-m}(\A^L) = \Expect[ r_{j-m}(\A^L) ] = \Expect\big[\sum_{F\in \mR_j(\A)} 1\{L\cap F\neq \zerovct\} \big] = \sum_{F\in \mR_j(\A)}\Prob\{ L\cap F\neq \zerovct\}
\end{equation*}
to express the sum of the Grassmann angles in terms of the number of regions of the reduced arrangement.
One can then derive the expression (for example, by applying Lemma~\ref{le:generic} iteratively),
 \begin{equation*}
  r_{j-m}(\A^{L}) = (-1)^{j-m}\left( a_{j0}+\cdots +a_{jm} + \sum_{k=m+1}^j (-1)^{k-m}a_{jk}\right)
 \end{equation*}
 to express the number of regions of the reduced arrangement in terms of the characteristic polynomial of $\A$.
 Via the Crofton formulas~\ref{cor:Croft}, we can use this to recover the expressions for the intrinsic volumes. 
\end{remark}

\subsection{Applications}

In this section we compute some examples and present some applications of Theorem~\ref{thm:klivans}.

\subsubsection{Product arrangements}

Let $\A,\B$ be two hyperplane arrangements in $\R^d$ and $\R^e$, respectively. The product arrangement in $\R^{d+e}$ is defined as
  \[ \A\times\B = \{H\times\R^e\mid H\in\A\}\cup\{\R^d\times H\mid H\in\B\} . \]
The characteristic polynomial is multiplicative, $\chi_{\A\times \B}(t)=\chi_\A(t)\chi_\B(t)$, and so is the bivariate polynomial~\eqref{eq:biv-char}, $X_{\A\times\B}(s,t) = X_\A(s,t)X_\B(s,t)$. This can either be shown directly~\cite[Ch.~2]{OT:92}, or deduced from Theorem~\ref{thm:klivans}, as the intrinsic volumes polynomial satisfies $P_{C\times D}(t)=P_C(t) P_D(t)$.

\subsubsection{Generic arrangements}\label{sec:gen-arr}

A hyperplane arrangement $\A$ is said to be in general position if the corresponding normal vectors are linearly independent.\footnote{We only discuss linear hyperplane arrangements; for generic affine hyperplane arrangements see for example~\cite{ardila2014algebraic}.} Combinatorial properties of such arrangements have been studied by Cover and Efron~\cite{cover1967geometrical}, who generalize results of Schl\"afli~\cite{schlafli} and Wendel~\cite{wend:62} to get expressions for, among other things, the average number of $j$-dimensional faces of a region in the arrangement. We set out to compute the characteristic polynomial of an arrangement of hyperplanes in general position, and in the process recover the formulas of Cover and Efron and a formula of Hug and Schneider~\cite{hug2015random} for the expected intrinsic volumes of the regions.

\begin{lemma}\label{lem:1}
Let $\A=\{H_1,\ldots,H_n\}$ be a generic hyperplane arrangement in~$\R^d$ with $n\geq d$. Then for $0<j\leq d$,
\begin{equation}\label{eq:chi_(A,j)-gen}
  (-1)^j \chi_{\A,j}(-t) = \binom{n}{d-j} \bigg( \binom{n-d+j-1}{j-1} + \sum_{k=1}^j \binom{n-d+j}{j-k} t^k\bigg) .
\end{equation}
\end{lemma}

\begin{proof}
Assume first that $j=d$. The proof in this case relies on Whitney's theorem~\cite[Prop.~3.11.3]{EC}
\begin{equation*}
  \chi_{\A}(t) = \sum_{\B\subseteq \A} (-1)^{|\B|} t^{d-\rho(\B)},
\end{equation*}
where $\rho$ denotes the rank of the arrangement $\B$. We can subdivide the sum into two parts:
\begin{equation*}
 \sum_{|\B|<d} (-1)^{|\B|} t^{d-\rho(\B)} + \sum_{|\B|\geq d} (-1)^{|\B|} t^{d-\rho(\B)}.
\end{equation*}
Since $\A$ is in general position, $\rho(\B)=|\B|$ if $|\B|\leq d$, and $\rho(\B)=d$ if $|\B|\geq d$. Collecting terms with equal rank, we obtain
\begin{equation*}
 \chi_{\A}(t) = \sum_{k=0}^{d-1} \binom{n}{k}(-1)^k t^{d-k} + \sum_{k=d}^n \binom{n}{k} (-1)^k .
\end{equation*}
An easy induction proof shows that $\sum_{k=d}^n \binom{n}{k} (-1)^k=\binom{n-1}{d-1}(-1)^d$, which settles the case $j=d$.

For the case $0<j<d$ note that if $L\in\mL_j(\A)$, then $L$ is the intersection of $d-j$ uniquely determined hyperplanes, and the restriction $\A^L$ is a generic hyperplane arrangement in~$L$ consisting of~$n-d+j$ hyperplanes. Furthermore, there are exactly $\binom{n}{d-j}$ linear subspaces of dimension~$j$ in $\mL(\A)$. Therefore, using the characterisation~\eqref{eq:expre-chi_(A,k)-chi_A} of the $j$th-level characteristic polynomial, we obtain
\begin{align*}
  (-1)^j \chi_{\A,j}(-t) & = \sum_{L\in \mL_j(\A)}(-1)^j\chi_{\A^L}(-t) = \binom{n}{d-j} \bigg( \binom{n-d+j-1}{j-1} + \sum_{k=1}^j \binom{n-d+j}{j-k} t^k\bigg) ,
\end{align*}
where the second equality follows from the case $j=d$.
\end{proof}

From Zaslavsky's Theorem~\ref{thm:zas1} we obtain from~\eqref{eq:chi_(A,j)-gen} the number of $j$-dimensional regions in a generic hyperplane arrangement, $r_j(\A)$, by setting $t=1$. Using the simplification
\begin{equation*}
 \binom{n-d+j-1}{j-1} + \sum_{k=1}^j \binom{n-d+j}{j-k} = 2 \ \sum_{k=1}^{j} \binom{n-d+j-1}{j-k}
\end{equation*}
we recognize the right-hand side as Schl\"afli's formula~\cite[(1.1)]{cover1967geometrical} for the number of regions of a generic arrangement of $n-d+j$ hyperplanes in $j$-dimensional space. The resulting formula for $r_j(\mathcal{A})$ allows us to recover the formula of Cover and Efron~\cite[Theorem 1]{cover1967geometrical} for the sum of the $f_j(C)$ over all regions.

If one takes one of these $j$-dimensional regions uniformly at random, then one also recovers the expression for the average number of $j$-dimensional faces from~\cite[Theorem 3']{cover1967geometrical}. Moreover,
then~\eqref{eq:chi_(A,j)-gen} and Theorem~\ref{thm:klivans} together yield a closed formula for the expected intrinsic volumes of the regions. In particular, the $d$-dimensional regions have expected intrinsic volumes of
\begin{align*}
  \Expect_{C\in \mR_d(\A)}[v_0(C)] & = \frac{1}{r_d(\A)}\binom{n-1}{d-1} , & \Expect_{C\in \mR_d(\A)}[v_k(C)] & = \frac{1}{r_d(\A)}\binom{n}{d-k} , \quad \text{if } k>0 .
\end{align*}
This is~\cite[Theorem 4.1]{hug2015random}.

\subsubsection{Braid and Coxeter arrangements}

Finally, we compute the $j$th-level characteristic polynomial for the three families of arrangements
\begin{align*}
   \A_A & := \big\{ \{\vct x\in\R^d\mid x_i=x_j\} \mid 1\leq i< j\leq d \big\} ,
\\ \A_{BC} & := \big\{ \{\vct x\in\R^d\mid x_i=\pm x_j\} \mid 1\leq i< j\leq d \big\}\cup \big\{ \{\vct x\in\R^d\mid x_i = 0\} \mid 1\leq i\leq d \big\} ,
\\ \A_D & := \big\{ \{\vct x\in\R^d\mid x_i=\pm x_j\} \mid 1\leq i< j\leq d \big\} .
\end{align*}
These arrangements are particularly nice to work with as the $d$-dimensional regions are all isometric; these chambers are indeed given by
\begin{align*}
   \A_A & :\; \{\vct x\in\R^d\mid x_{\pi(1)}\leq \dots\leq x_{\pi(d)}\} , & & \pi\in S_d ,
\\ \A_{BC} & :\; \{\vct x\in\R^d\mid 0\leq s_1 x_{\pi(1)}\leq \dots\leq s_d x_{\pi(d)}\} , & & s_1,\ldots,s_d\in \{\pm 1\}, \pi\in S_d ,
\\ \A_D & :\; \{\vct x\in\R^d\mid -s_1 x_{\pi(1)}\leq s_1 x_{\pi(1)}\leq \dots\leq s_d x_{\pi(d)}\} , & & s_1,\ldots,s_d\in \{\pm 1\}, \pi\in S_d .
\end{align*}
The characteristic polynomials of these arrangements are well known, see for example~\cite[Sec.~6.4]{ardila2014algebraic},
\begin{align}\label{eq:chi_A,A-BC-D}
   \chi_{\A_A}(t) & = \prod_{i=0}^{d-1} (t-i) , & \chi_{\A_{BC}}(t) & = \prod_{i=0}^{d-1} (t-2i-1) ,
\\\nonumber
   \chi_{\A_D}(t) & = (t-d+1)\prod_{i=0}^{d-2} (t-2i-1) = \chi_{\A_{BC}}(t) + d\prod_{i=0}^{d-2} (t-2i-1) .\hspace{-10cm}
\end{align}
The bivariate polynomial $X_{\A_A}(s,t)$ (along with affine generalizations) has been computed in~\cite[Thm.~8.3.1]{A:thesis}. We derive this again, along with polynomials for the other two arrangements, from the known characteristic polynomials.

\begin{lemma}
The $j$th-level characteristic polynomials for the above defined hyperplane arrangements are given by
\begin{align*}
   \chi_{\A_A,j}(t) & = \sttwo{d}{j} \prod_{i=0}^{j-1} (t-i) , & \chi_{\A_{BC},j}(t) & = \sttwo{d+1}{j+1} \prod_{i=0}^{j-1} (t-2i-1) ,
\\ \chi_{\A_D,j}(t) & = \chi_{\A_{BC},j}(t) + j\sttwo{d}{j} \prod_{i=0}^{j-2} (t-2i-1) , \hspace{-10cm}
\end{align*}
where $\sttwo{d}{j}$ denote the Stirling numbers of the second kind.
\end{lemma}

\begin{proof}
We first discuss the case $\A=\A_A$. From the formula for the chambers of $\A$ it is seen that an element in $\mL(\A)$ is of the form
  \[ L=\{\vct x\in\R^d\mid x_{\pi(k_1)}=\dots=x_{\pi(\ell_1)} ,\; x_{\pi(k_2)}=\dots=x_{\pi(\ell_2)} ,\; \dots \} , \]
where $k_1\leq \ell_1 < k_2\leq\ell_2 < \dots$. More precisely, for $L\in\mL_j(\A)$ there exists a unique partition $I_1,\ldots,I_j$, each nonempty, of $\{1,\ldots,d\}$ such that $L=\{\vct x\in\R^d\mid \forall i=1,\ldots,j, \forall a,b\in I_i, x_a=x_b\}$. The corresponding reduction $\A^L$ is easily seen to be a nonsingular linear transformation of the $j$-dimensional braid arrangement, so that $\chi_{\A^L}(t) = \prod_{i=0}^{j-1} (t-i)$. Since the number of partitions of $\{1,\ldots,d\}$ into $j$ nonempty sets is given by $\sttwo{d}{j}$, cf.~\cite{EC}, and by the characterisation~\eqref{eq:expre-chi_(A,k)-chi_A} of $\chi_{\A,j}(t)$, we obtain the claim in the case $\A=\A_A$.

In the case $\A=\A_{BC}$ we can argue similarly, but we need to keep in mind the extra role of the origin. For every element $L\in\mL(\A)$ there exists a subset $I$ of $\{1,\ldots,d\}$ of cardinality $|I|\geq j$, and a partition $I_1,\ldots,I_j$ of $I$ such that $L=\{\vct x\in\R^d\mid \forall a\not\in I, x_a=0 \text{ and } \forall i=1,\ldots,j, \forall a,b\in I_i, x_a=x_b\}$. The same argument as in the case $\A=\A_A$, along with the identity $\sum_{i=j}^d \binom{d}{i}\sttwo{i}{j} = \sttwo{d+1}{j+1}$, then settles the case $\A=\A_{BC}$.

In the case $\A=\A_D$ we have two types of linear subspaces:
\begin{align*}
   L_1 & =\{\vct x\in\R^d\mid x_{\pi(k_1)}=\dots=x_{\pi(\ell_1)} ,\; x_{\pi(k_2)}=\dots=x_{\pi(\ell_2)} ,\; \dots \} ,
\\ L_2 & =\{\vct x\in\R^d\mid 0=x_{\pi(k_1)}=\dots=x_{\pi(\ell_1)} ,\; x_{\pi(k_2)}=\dots=x_{\pi(\ell_2)} ,\; \dots \} .
\end{align*}
For the first type of linear subspace we obtain a reduction $\A^{L_1}$ that is isomorphic to the arrangement $\A_D$, while for the second type we obtain a reduction $\A^{L_2}$ that is isomorphic to the arrangement $\A_{BC}$ (each, of course, of the corresponding dimension). The number of subspaces of type $L_1$ is given by $\sttwo{d}{j}$ (as in the case $\A=\A_A$), while the number of subspaces of type $L_2$ is given by $\sttwo{d+1}{j+1}-\sttwo{d}{j}$ (as in the case $\A=\A_{BC}$, but noting that $|I|=d$ does not give a $BC$-type reduction). The same argument as before now yields the formula
\begin{align*}
   \chi_{\A_D,j}(t) & = \sttwo{d}{j} (t-j+1)\prod_{i=0}^{j-2} (t-2i-1) + \bigg( \sttwo{d+1}{j+1} - \sttwo{d}{j}\bigg) \prod_{i=0}^{j-1} (t-2i-1)
\\ & = \sttwo{d+1}{j+1} \prod_{i=0}^{j-1} (t-2i-1) + j\sttwo{d}{j} \prod_{i=0}^{j-1} (t-2i-1) ,
\end{align*}
which settles the case $\A=\A_D$.
\end{proof}

As before in the case of generic hyperplanes in Section~\ref{sec:gen-arr}, we finish by considering resulting formulas for uniformly random $j$-dimensional regions of the arrangement. We restrict to the arrangements $\A_A$ and $\A_{BC}$, and we restrict the formulas to the statistical dimensions. These statistical dimensions are particularly interesting for applications as seen in~\cite{edge}, where only the $d$-dimensional regions were considered. (Here, of course, the expectation vanishes since all $d$-chambers of these arrangements are isometric; for the lower-dimensional regions this is no longer true.)

Recall that the statistical dimension is given by $\sdim(C)=v_C'(1)$. Using again $r_j(\A)=(-1)^j\chi_{\A,j}(-1)$, we obtain
\begin{align*}
   \frac{1}{r_j(\A)}\sum_{F\in\mR_j(\A)} \sdim(F) & = \frac{1}{(-1)^j\chi_{\A,j}(-1)}\sum_{F\in\mR_j(\A)} v_F'(1) = -\frac{\chi_{\A,j}'(-1)}{\chi_{\A,j}(-1)} .
\end{align*}
We thus obtain:
\begin{align*}
   \chi_{\A_A,j}'(t) & = \chi_{\A_A,j}(t)\sum_{i=0}^{j-1} \frac{1}{t-i} , & \chi_{\A_{BC},j}'(t) & = \chi_{\A_{BC},j}(t)\sum_{i=0}^{j-1} \frac{1}{t-2i-1} ,
\\ -\frac{\chi_{\A_A,j}'(-1)}{\chi_{\A_A,j}(-1)} & = \sum_{i=0}^{j-1} \frac{1}{1+i} = H_j , & -\frac{\chi_{\A_{BC},j}'(-1)}{\chi_{\A_{BC},j}(-1)} & = \sum_{i=0}^{j-1} \frac{1}{1+2i+1} = \tfrac{1}{2}H_j ,
\end{align*}
where $H_j$ denotes the $j$th harmonic number.
We have thus derived the following application.

\begin{proposition}
Let $F_A\in\mR_j(\A_A)$ and $F_{BC}\in\mR_j(\A_{BC})$ be chosen uniformly at random among all elements in $\mR_j(\A_A)$ and $\mR_j(\A_{BC})$, respectively. Then their expected statistical dimensions are given by
\begin{align*}
  \Expect[\sdim(F_A)] & = H_j , & \Expect[\sdim(F_{BC})] & = \tfrac{1}{2}H_j .
\end{align*}
\end{proposition}

\bibliographystyle{alpha}
\bibliography{polycones_bib}

\end{document}